\documentclass[11pt]{amsart}
\usepackage{CJK}
\usepackage{indentfirst}
\usepackage{amscd}
\usepackage{amsmath, amssymb}
\usepackage{amsfonts}
\usepackage{enumerate}

\newcommand{\om}{M}
\newcommand{\vp}{\varphi}
\newcommand{\ve}{\varepsilon}
\newcommand{\ddbar}{\sqrt{-1} \partial \overline{\partial}}
\newcommand{\ol}{\overline}

\newcommand{\ri}{\rightarrow}

\begin{document}
\newcounter{theor}
\setcounter{theor}{1}
\newtheorem{claim}{Claim}
\newtheorem{theorem}{Theorem}[section]
\newtheorem{lemma}[theorem]{Lemma}
\newtheorem{corollary}[theorem]{Corollary}
\newtheorem{proposition}[theorem]{Proposition}
\newtheorem{prop}{Proposition}[section]
\newtheorem{question}{question}[section]
\newtheorem{defn}{Definition}[section]
\newtheorem{remark}{Remark}[section]

\numberwithin{equation}{section}

\title[deformed Hermitian-Yang-Mills equation]{The deformed Hermitian-Yang-Mills equation on almost Hermitian manifolds }
\author[L. Huang, J. Zhang and X. Zhang]{Liding Huang, Jiaogen Zhang and Xi Zhang}
\address{School of Mathematical Sciences, University of Science and Technology of China, Hefei 230026, P. R. China}
\email{huangld@mail.ustc.edu.cn}
\address{School of Mathematical Sciences, University of Science and Technology of China, Hefei 230026, P. R. China }
\email{zjgmath@mail.ustc.edu.cn}
\address{School of Mathematical Sciences, University of Science and Technology of China, Hefei 230026, P. R. China }
\email{mathzx@ustc.edu.cn}
\thanks {The authors are partially supported by NSF in China No.11625106,
11571332 and 11721101. The research was partially supported by the project ``Analysis and Geometry on Bundle" of Ministry of Science and Technology of the People's Republic of China, No.SQ2020YFA070080.}

\begin{abstract}
In this paper, we consider the deformed Hermitian-Yang-Mills equation on closed almost Hermitian manifolds. In the case of hypercritical phase,  we derive a priori estimates under the existence of an admissible $\mathcal{C}$-subsolution. As an application, we prove the existence of  solutions for the deformed Hermitian-Yang-Mills equation under the condition of existence of a supersolution. 
\end{abstract}
\maketitle

\section{Introduction}
 Motivated by Mirror Symmetry and Mathematical Physics, on a K\"{a}hler manifold $(M,\chi)$ of complex dimension $\dim_{\mathbb{C}}M=n$, the deformed Hermitian-Yang-Mills equation,  which has been studied extensively, can be written as the following form:
\begin{equation}\label{DHYM 1}
\textrm{Im}(\chi+\sqrt{-1}\omega_{u})^{n}=\tan(\hat{\theta})\textrm{Re}(\chi+\sqrt{-1}\omega_{u})^{n}
\end{equation}
where $\hat{\theta}$ is a constant and $\omega$ is a smooth real $(1,1)$-form,  $\omega_{u}=\omega+\ddbar{u}$. Assume that $(\lambda_{1}(u),\lambda_{2}(u), \cdots,\lambda_{n}(u))$ are the eigenvalues of $\omega_{u}$ with respect to $\chi.$ Without confusion, we also denote $\lambda_{i}(u)$ by $\lambda_{i}$, $i=1,\cdots, n.$ The equation (\ref{DHYM 1}) can be rewritten as
\begin{equation*}
\sum_{i}\arctan\lambda_{i}=\hat{\theta}.
\end{equation*}
By solving this equation, we can find a Hermitian metric on the line bundle over $M$ such that the argument of Chern curvature is constant \cite{JY}.

For the dimension $n=2$, one only need to solve a Monge-Amp\`ere equation \cite{JY}. In addition, Jacob-Yau \cite{JY}  used a parabolic flow to prove the existence of solution when $(M,\chi)$ has  non-negative othogonal bisectional curvature and $\hat{\theta}$ satisfies the hypercritical phase condition, i.e., $n\frac{\pi}{2}> \hat{\theta}>(n-1)\frac{\pi}{2}$ for general dimensions. V. Pingali \cite{PV,PV1} proved the existence of solution when $n=3$. 
 Collins-Jacob-Yau \cite{CJY} gave the existence theorem of \eqref{DHYM 1}  under the condition of existence of subsolution in general dimensions. R. Takahashi \cite{TR}  introduced the tangent Lagrangian phase flow and used it to proved the existence of solution of the deformed-Hermitian-Yang-Mills equation, assuming the existence
 of a $\mathcal{C}$-subsolution.  For more details, we refer to \cite{CG19,CG20,CCL20,CY,CS20,CXY, DP} and the references therein.

In this paper, we give a priori estimates on an almost Hermitian manifold $(M, \chi, J)$ with real dimension $2n$ for the deformed Hermitian-Yang-Mills equation. 
Consider the following general equation
\begin{equation}\label{DHYM}
F(\omega_{u})=f(\lambda_{1},\cdots, \lambda_{n})=\sum_{i}\arctan\lambda_{i}=h.
\end{equation}
where $h:M\rightarrow ((n-1)\frac{\pi}{2}, n\frac{\pi}{2})$ is a given function on $M$ and $(\lambda_{1}, \cdots,\lambda_{n}) $ are the eigenvalues of $\omega_{u}$ with respect to $\chi$. 

We now state our main result.
\begin{theorem}\label{main theorem}
	Assume $h:M\rightarrow ((n-1)\frac{\pi}{2}, \frac{n\pi}{2})$ is a smooth function and $\underline{u}:M\rightarrow \mathbb{R}$ is a smooth $\mathcal{C}$-subsolution (Definition \ref{sub}). Suppose $\omega$ is a smooth  real $(1, 1)$-form. Let $u$ be a solution of \eqref{DHYM}, for each $0<\beta<1$, then we have
	\begin{equation}
	\| u\|_{C^{k,\beta}}\leq C,
	\end{equation}
	where $C$ depends on $\|h\|_{C^{k+1}(M)}$, $\inf_{M} h$, $\underline{u}$, $k$, $\beta$, $(M,\chi, J)$ and $\omega$.
\end{theorem}
\bigskip

Almost Hermitian manifolds have been studied extensively motivated by differential geometry and mathematical physics (see \cite{DeT06,GSVY90,HL15,NH03,ST12} and references therein).
The theory of fully nonlinear elliptic equations were developed, such as \cite{CTW16, CHZ17, ZJ}, etc.
On the other hand, the geometry of complex vector bundles over almost Hermitian manifolds were researched in \cite{BT96, WZ, ZX}  and references therein.
Note that Lau-Tseng-Yau \cite{LTY} studied SYZ mirror symmetry in the context of non-K\"ahler Calabi-Yau manifolds. Generalized complex geometry, proposed by Hitchin \cite{HL15}, is closed related to flux compactifications in string theory \cite{GM}.  In fact, the deformed Hermitian-Yang-Mills equation plays an important role in mirror symmetry and  string theory \cite{CXY}.  It is natural to
consider the equation \eqref{DHYM} to non-K\"{a}hler case. Motivated by these works, we prove this result.


In order to prove the theorem above, we will use the maximum principle. A crucial ingredient of the proof need the equation \eqref{DHYM} is concave. To prove that the equation is  concave, we need to assume $h$ satisfies the hypercritical phase condition, i.e., $n\frac{\pi}{2}> \hat{\theta}>(n-1)\frac{\pi}{2}$. In addition, to use the Proposition \ref{prop subsolution} provided by $\mathcal{C}$-subsolution, the equation have to satisfy some properties in Lemma \ref{lemma subsolution}. We can obtain these properties when $h$ satisfies the hypercritical phase condition.



In the proof of the second order estimates, we apply the maximum  principle. We will use the argument of \cite{CTW16} and \cite{ZJ}.
To deal with the bad third order terms, we need to give a lower bound  for the  third order terms from the concavity of the equation (see Lemma \ref{lower bound of concave}). 
In the proof, we give a positive lower  bound for the complex eigenvalues $\omega_{u}$ provided by $\inf_{M} h>\frac{(n-1)\pi}{2}$.


When $h\in ((n-1)\frac{\pi}{2}, n\frac{\pi}{2})$ is a constant, we can  prove the existence under the condition of existence of supersolution.
\begin{theorem}\label{constant angle theorem}
	Assume there exists a supersolution $\hat{u}$ of \eqref{DHYM}, i.e. $F(\omega_{\hat{u}})\leq h$. Suppose $F(\omega_{\hat{u}})>\frac{(n-1)\pi}{2}$ and $h\in (\frac{(n-1)\pi}{2}, \frac{n\pi}{2})$ is a constant. 	If there exists a  $\mathcal{C}$-subsolution $\underline{u}$ for \eqref{DHYM}, we have a function $u$ and a constant $c$ such that
		\begin{equation*}
	\sum_{i}\arctan \lambda_{i}=h+c,
	\end{equation*}
	where $	h+c>\frac{n-1}{2}\pi$.
\end{theorem}

\begin{remark}
	If $(M,\chi)$ is a K\"{a}hler manifold, then the supersolution $\hat{u}$ is a solution of the equation \eqref{DHYM}, using $\hat{\theta}$ is an invariant. However, $\hat{\theta}$ is not an invariant on almost Hermitian manifolds.
\end{remark}
 Let $\pi:L\rightarrow M $ be a complex line bundle on $M$
and $\varpi$ be a Hermitian metric on $L$ (for more details, see \cite[Section 2]{BT96} or \cite{WZ}). There exists a unique type $(1, 0)$
Hermitian connection $D_{\varpi}$ which is called the canonical Hermitian connection. Let $F(\varpi)$ be the curvature form of connection $D_{\varpi}$ and $F^{1,1}(\varpi)$ be the $(1,1)$ part of the curvature of $\varpi$. Denote $\varpi(u)=e^{-u}\varpi$. We have (see \cite{BT96} or \cite[(2.5)]{WZ})
\[F^{1,1}(\varpi(u))=F^{1,1}(\varpi)+\partial\bar{\partial} u. \]
We assume $\int_{M}\chi^{n}=1.$
Set $\omega=\sqrt{-1} F^{1,1}(\varpi)$ and
$$\hat{\theta}(\varpi(u))=\text{Arg}\int_{M}\frac{(\chi+\sqrt{-1}(\omega+\ddbar u))^{n}}{\chi^{n}}\chi^{n}.$$
 Here $\text{Arg}\, \vp$ means the argument of a complex function $\vp$.
%
Note that \eqref{DHYM 1} is equivalent to \eqref{DHYM}. We immediately obtain

\begin{corollary}
 	Suppose there exists a supersolution $\hat{u}$ of \eqref{DHYM}. Assume $F(\omega_{\hat{u}})>\frac{(n-1)\pi}{2}$ and $h\in (\frac{(n-1)\pi}{2}, \frac{n\pi}{2})$ is a constant.
	 Suppose there exists $\varpi(\underline{u})=e^{-\underline{u}}\varpi$ such that $\underline{u}$ is a $\mathcal{C}$-subsolution of \eqref{DHYM}.
	There exists a Hermitian metric $\varpi(u)$ on line bundle $L$ such that the argument of $\frac{\big(\chi+\sqrt{-1}F^{1,1}((\varpi)(u))\big)^{n}}{\chi^{n}}$ is constant.
\end{corollary}
In K\"{a}hler case, the proof in \cite{CJY} relied on the argument $\hat{\theta}$ of the integral is independent of the choice of $\omega_{u}$. They used this fact to prove the $\mathcal{C}$-subsolution is preserved along the family of equations used in the continuity method.
However, $\hat{\theta}(\varpi(u))$ depend on $u$ in our case. Under the existence of supersolution, we prove the $\mathcal{C}$-subsolution and hypercritical condition are preserved when we use the continuity method, by using the arguments of \cite{CJY} and \cite{SW}.

The organization of the paper is as follows.
In section 2, we recall the definition of $\mathcal{C}$-subsolution and some properties of equation \eqref{DHYM}. In section 3, we give the $C^{0}$ estimates. We use the argument of Sz\'{e}kelyhidi \cite{SG} (see also \cite{BL05}).   In section 4, the gradient estimates are proved. In Section 5, we will give the  second order estimates and complete the proof of Theorem \ref{main theorem}. In section 6, by the continuity method, we  prove Theorem \ref{constant angle theorem} under the condition of the existence of the supersolution.

\vskip3mm
{\bf Acknowledgment: }  The authors  would like to thank  the referees for many useful  suggestions and comments.


\section{preliminaries}

On an  almost Hermitian manifold  $(M,\chi,J)$ with  real dimension $2n$,  for any $(p,q)$-form $\beta$,    we can define $\partial$ and $\overline{\partial}$ operators  (cf. \cite{HL15,CTW16}). Denote by $A^{1,1}(M)$ the space of smooth real (1,1) forms on $(M,\chi,J)$. Then,  for any  $\varphi\in C^{2}(M)$,  $\ddbar \varphi=\frac{1}{2}(dJd\varphi)^{(1,1)}$ is a real $(1,1)$-form in $A^{1,1}(M)$. Let $\{e_{i}\}_{i=1}^{n}$ be a local frame for $T_{\mathbb{C}}^{(1,0)}M$ and $\{\theta^{1}, \cdots, \theta^{n}\}$  be a dual coframe associated to the  metric $\chi$ on $(M,\chi,J)$(cf. \cite[(2.5)]{HL15}).  Denote $\chi_{i\ol{j}}=\chi(e_{i},\ol{e}_{j})$ and $g_{i\ol{j}}=\omega(e_{i},\ol{e}_{j})$. Then $\chi=\chi_{i\ol{j}}\sqrt{-1}\theta^{i}\wedge \ol{\theta}^{j}$ and $\omega=g_{i\ol{j}}\sqrt{-1}\theta^{i}\wedge \ol{\theta}^{j}$.
We have
\begin{equation*}
\varphi_{i\overline{j}}=(\ddbar \varphi)(e_{i},\overline{e}_{j})=e_{i}\overline{e}_{j}(\varphi)-[e_{i},\overline{e}_{j}]^{(0,1)}(\varphi),
\end{equation*}
where $[e_{i}, \bar{e}_{j}]^{(0,1)}$ is the $(0,1)$ part of the Lie bracket $[e_{i}, \bar{e}_{j}]$.
We use the following notation
\begin{equation*}
F^{i\overline{j}}=\frac{\partial
	\sum_{k}\arctan\lambda_{k}(\tilde{g})}{\partial \tilde{g}_{i\overline{j}}},
\end{equation*}
where $\tilde{g}_{i\bar{j}}=g_{i\bar{j}}+u_{i\bar{j}}.$
For any point $x_{0}\in M$, let $\{e_{i}\}_{i=1}^{n}$ be a local unitary frame (with respect to $\chi$) such that $\tilde{g}_{i\overline{j}}(x_{0})=\delta_{ij}\tilde{g}_{i\overline{i}}(x_{0})$. We denote $\tilde{g}_{i\overline{i}}(x_{0})$ by $\lambda_{i}$.
It is useful to order $\{\lambda_{i}\}$ such that
\begin{equation*}
\lambda_{1}\geq\lambda_{2}\geq\cdots\geq\lambda_{n}.
\end{equation*}
Then at $x_{0}$, we have
\begin{equation}\label{the first order derivative of F}
F^{i\overline{j}}=F^{i\overline{i}}\delta_{ij}=\frac{1}{1+\lambda_{i}^{2}}\delta_{ij}.
\end{equation}
By \cite[(66)]{SG} or \cite{GC}, we deduce
\[F^{i\bar{k},j\bar{l}}=f_{ij}\delta_{ik}\delta_{jl}
+\frac{f_{i}-f_{j}}{\lambda_{i}-\lambda_{j}}
(1-\delta_{ij})\delta_{il}\delta_{jk}.\]
It follows that, at $x_{0}$,
\begin{equation}\label{the second order derivative of F}
F^{i\overline{k},j\overline{l}}=
\left\{\begin{array}{ll}
F^{i\overline{i},i\overline{i}}, \text{~~~~if $i=j=k=l$;}\\[1mm]
F^{i\overline{k},k\overline{i}}, \text{~~~~if $i=l$, $k=j$, $i\neq k$;}\\[1mm]
0,  \text{\quad\quad\quad~otherwise.}
\end{array}\right.
\end{equation}
Moreover, at $x_{0}$,
\begin{equation}\label{Definition of Fijkl}
\begin{split}
F^{i\overline{i},i\overline{i}} & = -\frac{2\lambda_{i}}{(1+\lambda^{2}_{i})^{2}},\\
F^{i\overline{k},k\overline{i}} & = -\frac{\lambda_{i}+\lambda_{k}}{(1+\lambda^{2}_{i})(1+\lambda^{2}_{k})}.
\end{split}
\end{equation}

The linearization operator of \eqref{DHYM} is
\begin{equation}\label{L}
L:=\sum_{i,j}F^{i\bar{j}}(e_{i}\bar{e}_{j}-[e_{i},\bar{e}_{j}]^{0,1}).
\end{equation}
Note that $[e_{i},\bar{e}_{j}]^{0,1}$ are first order defferential operators. By \eqref{the first order derivative of F}, $L$ is a second order elliptic operator.
\subsection{$\mathcal{C}$-subsolution}
Now we recall the definition of $\mathcal{C}$-subsolution of \eqref{DHYM} (\cite{CJY}, \cite{SG}).
Denote  \begin{equation*}
\Gamma_{n}=\{\lambda=(\lambda_{1}, \cdots, \lambda_{n})\in \mathbb{R}^{n},\  \lambda_{i}>0, 1\leq i \leq n\},
\end{equation*}
\begin{equation*}
\Gamma=\{\lambda=(\lambda_{1}, \cdots, \lambda_{n})\in \mathbb{R}^{n},\  \sum_{i}\arctan(\lambda_{i})> (n-1)\frac{\pi}{2}\},
\end{equation*}
and
\begin{equation*}
\Gamma^{\sigma}=\{\lambda\in \Gamma, \sum_{i}\arctan(\lambda_{i})> \sigma\},
\end{equation*}
where $\sigma\in((n-1)\frac{\pi}{2}, n\frac{\pi}{2})$.
\begin{defn}[\cite{CJY,SG}]\label{sub}
	We say that a smooth function $\underline{u}:M\rightarrow R$ is a $\mathcal{C}$-subsolution of \eqref{DHYM} if at each point $x\in M$, we have
	\begin{equation*}
	\left\{\lambda\in \Gamma: \sum_{i=1}^{n}\arctan(\lambda_{i})=h(x), \text{\ and\ }  \lambda-\lambda(\underline{u})\in \Gamma_{n}\right\}
	\end{equation*}
	is bounded.
\end{defn}
Collins-Jacob-Yau gave an explicit description of $\mathcal{C}$-subsolution.
\begin{lemma}[\cite{CJY}, Lemma 3.3]\label{subsolution}
	A smooth function $\underline{u}: M\rightarrow \mathbb{R} $ is a $\mathcal{C}$-subsolution of \eqref{DHYM} if and only if at each point $x\in M$,  for all $j=1,\cdots, n$, we have
	\begin{equation*}
	\sum_{i\neq j }\arctan (\lambda_{i}(\underline{u}))>h(x)-\frac{\pi}{2},
	\end{equation*}
	where $\lambda_{1}(\underline{u}), \cdots, \lambda_{n}(\underline{u})$ are the eigenvalues of $\omega_{\underline{u}}$ with respect to $\chi$.
\end{lemma}
Therefore, there are uniform constants $\delta, R>0$ such that at each $x\in M$ we have
\begin{equation}\label{2..12}
(\lambda(\underline{u})-\delta\textbf{1}+\Gamma_{n})\cap \partial\Gamma^{h(x)}\subset B_{R}(0),
\end{equation}
where $B_{R}(0)$ is a $R$-radius ball in $\mathbb{R}^{n}$ with center 0,  $\textbf{1}=(1,1,\cdots,1)$.

We now prove the following lemma:
\begin{lemma}\label{lemma subsolution}
	Suppose $h\in ((n-1)\frac{\pi}{2},n\frac{\pi}{2})$, then we have the following properties:
	\begin{enumerate}
		\item $f_{i}=\frac{\partial f}{\partial\lambda_{i}}>0$ for all $i$, and  the equation \eqref{DHYM} is concave,
		\item $\sup_{\partial \Gamma} f<\inf_{M}h$,
		\item for any $\sigma<\sup_{\Gamma} f$ and $\lambda\in \Gamma$ we have $\underset{t\rightarrow \infty}{\lim}f(t\lambda)>\sigma.$
	\end{enumerate}
\end{lemma}
\begin{proof}
	Note $f_{i}=\frac{1}{1+\lambda_{i}^{2}}$.	It is obvious that $f_{i}>0$ and  $\sup_{\partial \Gamma} f<\inf_{M}h$, if $h>(n-1)\frac{\pi}{2}$.
	
	When $h>(n-1)\frac{\pi}{2}$, we have
	\begin{equation}\label{psh}
	\lambda_{i}>0, ~ \textrm{for}~ i=1,2,\cdots,n.
	\end{equation} In fact, if there is $\lambda_{i}\leq 0$, then we must have $\sum_{j}\arctan\lambda_{j}\leq (n-1)\frac{\pi}{2} $ and this is a contradiction. Hence, by \eqref{the second order derivative of F},  $f$  is  concave and
	$\underset{t\rightarrow \infty}{\lim}f(t\lambda)=n\frac{\pi}{2}>\sigma$.
\end{proof}
Using the above Lemma and \cite[Proposition 6 and Lemma 9]{SG}, we have the following Proposition. It plays an important role in the proof.

\begin{proposition}\label{prop subsolution}
	Let $[a,b]\subset((n-1)\frac{\pi}{2}, n\frac{\pi}{2})$ and $\delta, R>0$. There exists $\theta>0$, depending on $\sigma$ and the set in \eqref{set subsolution}, with the following property: suppose that $\sigma\in [a,b]$ and
	$B$ is a Hermitian matrix such that
	\begin{equation}\label{set subsolution}
	(\lambda(B)-2\delta\textbf{1}+\Gamma_{n})\cap\partial\Gamma^{\sigma}\subset B_{R}(0).
	\end{equation}
	Then for any Hermitian matrix A with eigenvalues $\lambda(A)\in \partial \Gamma^{\sigma}$ and $|\lambda(A)|>R$, we either have
	\begin{equation}\label{sub1}
	\sum_{p,q}F^{p\ol{q}}(A)[B_{p\ol{q}}-A_{p\ol{q}}]>\theta\sum_{p}F^{p\ol{p}}(A)
	\end{equation}
	or
	\begin{equation}\label{sub2}
	F^{i\ol{i}}(A)>\theta\sum_{p}F^{p\ol{p}}(A)
	\end{equation} for all $i$. In addition, there exists a constant $\mathcal{K}$ depending on $\sigma$ such that
	\begin{equation}\label{sub3}
	\sum_{i}F^{i\ol{i}}>\mathcal{K}.
	\end{equation}
\end{proposition}
\begin{corollary}\label{Cor2.6}
Suppose $h\in((n-1)\frac{\pi}{2},n\frac{\pi}{2})$. Assume $\underline{u}$ is an admissible $\mathcal{C}$-subsolution and $u$ is the smooth solution for (\ref{DHYM}). Then
there exists a constant $\theta>0$ (depending only on $h$ and $\underline{u}$) such that 
 either
\begin{equation}\label{2..25}
L(\underline{u}-u)\geq \theta\sum_{i}F^{i\bar{i}}(\tilde{g})
\end{equation}
or
\begin{equation}\label{2..26}
F^{k\bar{k}}(\tilde{g})\geq \theta\underset{i}{\sum}F^{i\bar{i}}(\tilde{g}),~ \textrm{for}~ k=1,2,\cdots,n
\end{equation}
 if  $\lambda(\omega_{u})\in(\lambda(\omega_{\underline{u}})-\delta \textbf{1}+\Gamma_{n})\cap \partial\Gamma^{h}$.

In addition, there is a constant $\mathcal{K}>0$ depending on $h$ and $\underline{u}$ such that
  \begin{equation}\label{2..27}
  \mathcal{F}:=\sum_{i}F^{i\bar{i}}(\tilde{g})>\mathcal{K},~ \textrm{if}~\lambda(\omega_{u})\in\partial\Gamma^{h}.
  \end{equation}
\end{corollary}

\begin{proof}
By Definition \ref{sub}, 	there are uniform constants $\delta, R>0$ such that at each $x\in M$ we have
	\begin{equation*}
	(\lambda(\underline{u})-\delta\textbf{1}+\Gamma_{n})\cap \partial\Gamma^{h(x)}\subset B_{R}(0).
	\end{equation*}
If $|\lambda(u)|>R$, by Proposition \ref{prop subsolution}, the results follow. If  $|\lambda(u)|\leq R$, then $1\geq F^{i\bar{i}}\geq \frac{1}{1+R^{2}}, i=1,2, \cdots, n$, which implies \eqref{2..26} and \eqref{2..27} hold.

\end{proof}

\section{Zero order estimates}
In this seciton we prove the $C^{0}$ estimates.
We need the following Proposition provided by \cite[Proposition 2.3]{CTW16}.
\begin{proposition}\label{L1}
	Let $(M, \chi, J)$ be a compact almost Hermitian manifold.  Suppose that $\psi$ satisfies
	\begin{equation*}
	\omega+\ddbar \psi>0,\ \  \sup_{M}\psi=0.
	\end{equation*}
	Then there exists a constant $C$ depending only on $(M, \chi, J)$ and $\omega$ such that
	\begin{equation*}
	\int_{M}(-\psi)\chi^{n}\leq C.
	\end{equation*}
\end{proposition}	
\begin{proof}
	There exists a constant $C_{0}$ such that $C_{0}\chi\geq \omega.$ Therefore, by \eqref{psh}, $C_{0}\chi+\ddbar\psi >0$. Then by \cite[Proposition 2.3]{CTW16}, we  have
	\begin{equation*}
	\int_{M}(-\psi)\chi^{n}\leq C.
	\end{equation*}
\end{proof}	
Indeed, by \eqref{psh}, the assumption in Proposition \ref{L1} is satisfied in our paper.
The following variant of the Alexandroff-Bakelman-Pucci maximum principle \cite[Proposition 11]{SG}, similarly with Gilbarg-Trudinger \cite[Lemma 9.2]{GT01},  is used to prove the $C^{0}$ estimates.
\begin{proposition}\label{ABP}
	Let $\vp:B_{1}(0)\rightarrow \mathbb{R}$ be a smooth function, such that $\vp(0)+\ve\leq \inf_{\partial B_{1}(0)}\vp$, where $\ve>0$, $B_{1}(0)\subset \mathbb{R}^{2n}$. Define the set
	\begin{equation*}
	P=\left\{ x\in B_{1}(0):
	\begin{matrix} |D\vp(x)|<\frac{\ve}{2}, \text{\ and\ }
	\vp(y)\geq \vp(x)+D\vp(x)\cdot (y-x)\\
	\text{ for all } y\in B_{1}(0)
	\end{matrix} \right\}.
	\end{equation*}
	Then there exists a constant $c_{0}$ depending only on $n$ such that
	\begin{equation*}
	c_{0}\ve^{2n}\leq \int_{P}\det(D^{2}\vp).
	\end{equation*}
\end{proposition}

The $C^{0}$ estimates follow the argument of \cite[Proposition 10]{SG}  or \cite[Proposition 3.1]{CTW16}. It  is similar to \cite[Proposition 3.2]{ZJ}. For the reader's convenience, we include the proof.
\begin{proposition}\label{Prop3.2}
	Let $u$ be the solution for \eqref{DHYM} with $\sup_{M}(u-\underline{u})=0$. Then
	\begin{equation}\label{}
	\|u\|_{L^{\infty}}\leq C
	\end{equation}
	for some constant $C>0$ depending on $(M, \chi, J)$, $\omega$, $h$ and $\underline{u}$.
\end{proposition}
\begin{proof}
	From the hypothesis, it suffices to estimate the infimum $m_{0}=\inf_{M}(u-\underline{u})$. We may assume $m_{0}$ is attained at $x_{0}$. Choose a local coordinate chart $(x^{1},\cdots, x^{2n})$ in a neighborhood of $x_{0}$ containing the unit ball $B_{1}(0)\subset \mathbb{R}^{2n}$ such that the coordinates of $x_{0}$ are the origin  $0\in \mathbb{R}^{2n}$.
	
	Consider the test function
	\[
	v:= u-\underline{u}+\ve\sum_{i=1}^{2n}(x^{i})^{2}
	\]
	for a small $\ve>0$ determined later. Then we have
	\[v(0)=m_{0}~ \textrm{and}~v\geq m_{0}+\ve ~\textrm{on}~ \partial B_{1}(0).\]
	We define the lower contact set of $v$ by
	\begin{equation}\label{P}
	P:=\Big\{x\in B_{1}(0): |Dv(x)|\leq \frac{\ve}{2},
	v(y)\geq v(x)+Dv(x)\cdot (y-x),\text{\ for all\ } y\in B_{1}(0)\Big\}.
	\end{equation}
	By Proposition \ref{ABP}, we have
	\begin{equation}\label{3.5}
	c_{0}\ve^{2n}\leq \int_{P}\det(D^{2}v).
	\end{equation}
Let $\Big(D^{2}(u-\underline{u})\Big)^{J}$ be the $J$-invariant part of $\Big(D^{2}(u-\underline{u})\Big)$, i.e.,
	\begin{equation}\label{zero order D2 2}
	\Big(D^{2}(u-\underline{u})\Big)^{J}=\frac{1}{2}(D^{2}(u-\underline{u})+J^{T}D^{2}(u-\underline{u})J),
	\end{equation}
	where $J^{T}$ is the transpose of  $J$. Note that $0\in P$ and $D^{2}v\geq 0$ on $P$. Then we deduce
	\begin{equation}\label{zero estimate D2}
	\left(D^{2}(u-\underline{u})\right)^{J}(x)\geq (D^{2}v)^{J}(x)-C\ve Id\geq -C\ve Id, \text{\ for\ }  x\in P.
	\end{equation}
	Consider the bilinear form $H(v)(X,Y)=\ddbar v(X, JY)$. In fact, we obtain
	\[H(v)(X,Y)(x)=\frac{1}{2}\left(D^{2}v\right)^{J}(x)+E(v)(x),\ \  x\in M ,\]	
	where  $E(v)(x)$ is an error matrix which depends linearly on $Dv(x)$ (see e.g. \cite[P. 443]{TWWY15}).	
	Using $|D(u-\underline{u})|\leq \frac{5\ve}{2}$ on $P$ and \eqref{zero estimate D2}, it follows
	\[
	\begin{split}
	H(u)-H(\underline{u})=&\Big(D^{2}u+E(Du)\Big)^{J}-\Big(D^{2}\underline{u}
	+E(D\underline{u})\Big)^{J}  \\
	=&\Big(D^{2}(u-\underline{u})\Big)^{J}+ \Big(E(D(u-\underline{u}))\Big)^{J}\\
	\geq& -C\ve Id.
	\end{split}
	\]
	Hence
	\[
	\omega_{u}-\omega_{\underline{u}}\geq -C\ve\chi.
	\]
	Therefore, if we choose $\ve$ sufficient small such that $C\ve\leq \delta$, then
	\[
	\lambda(u)\in \lambda(\underline{u})-\delta \textbf{1}+\Gamma_{n}.
	\]
	On the other hand, the equation \eqref{DHYM} implies $\lambda(u)\in \partial\Gamma^{h}.$
	Consequently, \[\lambda(u)\in (\lambda(\underline{u})-\delta \textbf{1}+\Gamma_{n})\cap \partial\Gamma^{h} \subset B_{R}(0)\] for some $R>0$ by the argument (\ref{2..12}). This gives an upper bound for $H(u)$ and hence also for $H(v)-E(Dv)$ on $P$.
	
	Note $\det(A+B)\geq \det(A)+\det(B)$ for positive definite Hermitian matrices $A, B$. Recall the definition of $(D^{2}v)^{J}$ in \eqref{zero order D2 2}. Then on $P$, we have
	\begin{equation}\label{3.10}
	\begin{split}
	\det(D^{2}v)\leq & 2^{2n-1}\det((D^{2}v)^{J}) \\
	=  & 2^{2n-1}\det(H(v)-E(Dv))\leq C.
	\end{split}
	\end{equation}
	Plugging (\ref{3.10}) into (\ref{3.5}), we obtain
	\begin{equation}\label{zero estimates 3}
	c_{0}\ve^{2n}\leq C|P|.
	\end{equation}
	For each $x\in P$, choosing $y=0$ in $\eqref{P}$, we have
	\begin{equation*}
	m_{0}=v(0)\geq v(x)-|Dv(x)||x|\geq v(x)-\frac{\ve}{2}.
	\end{equation*}
	We may and do assume $m_{0}+\ve\leq 0$ (otherwise we are done), then, on $P$,
	\[
	-v\geq |m_{0}+\ve|.
	\]
Integrating it on $P$,   we get
	
	\[|P|\leq \frac{\int_{P}(-v)\chi^{n}}{|m_{0}+\ve|}\leq \frac{C}{|m_{0}+\ve|},\]
	where in the last inequality we used  Proposition \ref{L1}.
	By \eqref{zero estimates 3}, we get a uniform lower bound for $m_{0}$.
\end{proof}

\section{First order estimate}
In this section,we give the proof of the $C^{1}$ estimates. Let $|\nabla u|_{\chi}$ be the norm of gradient $u$ with respect to $\chi$. For convenience, we use $|\nabla u|$ to denote $|\nabla u|_{\chi}.$ We denote $\mathcal{F}=\sum_{i}F^{i\bar{i}}.$
\begin{proposition}\label{estimates of gradient}
	\begin{equation}
	|\nabla u|\leq C
	\end{equation}
	for some constant $C$ depending on $(M ,\chi, J)$, $\omega$, $\|h\|_{C^{1}}$ and $\underline{u}$.
\end{proposition}
~\\ \textit{Proof.}
Let $\zeta= Ae^{B_{1}\eta}$ where  \[\eta=\underline{u}-u-\inf_{M}(\underline{u}-u)\geq0,\] $A, B_{1}$ are positive constants to be chosen later. Consider the test function
\[\hat{Q}:= e^{\zeta}|\nabla u|^{2}.\]
Suppose $\hat{Q}$ achieves maximum at the point $x_{0}$. Near $x_{0}$, there exists a local unitary frame $\{e_{i}\}_{i=1}^{n}$ (with respect to $\chi$) such that at $x_{0}$, we have
\begin{equation*}
\text{$\chi_{i\bar{j}}=\delta_{ij}$, $\tilde{g}_{i\bar{j}}=\delta_{ij} \tilde{g}_{i\bar{j}}$ and $\tilde{g}_{1\bar{1}}\geq \tilde{g}_{2\bar{2}}\geq\cdots\geq \tilde{g}_{n\bar{n}}$.}
\end{equation*}
%

From now on, we will use the Einstein summation convention, and all the following calculations are done at $x_{0}$. The $C$ below in this section denote the constants that may change from line to line, where $C$ depends on all the allowed data that we determined later.

Recall that $L$ is defined in \eqref{L}. By the maximum principle, it follows that
\begin{equation}\label{3.9}
\begin{split}
0\geq \frac{L(\hat{Q})}{B_1\zeta e^{\zeta}|\nabla u|^{2}}=&\frac{L(|\nabla u|^{2})}{B_1\zeta|\nabla u|^{2}}+\frac{L(e^{\zeta})}{B_1\zeta e^{\zeta}}+2F^{i\bar{i}}\textrm{Re}\Big\{e_{i}(\zeta)
\frac{\bar{e}_{i}(|\nabla u|^{2})}{B_1\zeta|\nabla u|^{2}}\Big\}\\
=&\frac{L(|\nabla u|^{2})}{B_1\zeta|\nabla u|^{2}}+L(\eta)+B_{1}(1+\zeta)F^{i\bar{i}}|e_{i}(\eta)|^{2}\\
&+\frac{2}{|\nabla u|^{2}}\sum_{j}F^{i\bar{i}}\textrm{Re}\Big\{e_{i}(\eta)\bar{e}_{i}e_{j}(u)\bar{e}_{j}(u)
+e_{i}(\eta)\bar{e}_{i}\bar{e}_{j}(u)e_{j}(u)\Big\}.
\end{split}
\end{equation}

Now we deal with  these terms in turn. First we have
\begin{lemma}\label{calculation  of gradient}
	\begin{equation*}
	L(|\nabla u|^{2})\geq 2\sum_{j}\textrm{Re}\{e_{j}(h)\bar{e}_{j}u\}+ (1-\varepsilon) \sum_{j}F^{i\bar{i}}(|e_{i}e_{j}u|^{2}+|e_{i}\bar{e}_{j}u|^{2})- \frac{C}{\varepsilon} \mathcal{F}|\nabla u|^{2}.
	\end{equation*}
\end{lemma}
\begin{proof}
	By direct calculation,
	\begin{equation}\label{calculation gradient}
	L(|\nabla u|^{2})=F^{i\bar{i}}\Big({e_{i}e_{\bar{i}}}(|\nabla u|^{2})-[e_{i},\bar{e}_{i}]^{0,1}(|\nabla u|^{2})\Big):= I+II+III,
	\end{equation}
	where
	\[I=\sum_{j}F^{i\bar{i}}(e_{i}\bar{e}_{i}e_{j}u-[e_{i},\bar{e}_{i}]^{0,1}e_{j}u)\bar{e}_{j}u;\]
	\[II=\sum_{j}F^{i\bar{i}}(e_{i}\bar{e}_{i}\bar{e}_{j}u-[e_{i},\bar{e}_{i}]^{0,1}\bar{e}_{j}u)e_{j}u;\]
	\[III=\sum_{j}F^{i\bar{i}}(|e_{i}e_{j}u|^{2}+|e_{i}\bar{e}_{j}u|^{2}).\]
	Differentiating \eqref{DHYM} along $e_{j}$ without summation, we have
	\[ F^{i\ol{i}}e_{j}(g_{i\ol{i}})+F^{i\bar{i}}(e_{j}e_{i}\bar{e}_{i}u-e_{j}[e_{i},\bar{e}_{i}]^{0,1}u)=e_{j}(h).
	\]
	Recall the definition of Lie bracket $[e_{i},e_{j}]=e_{i}e_{j}-e_{j}e_{i}$. Then we have
	\begin{equation}\label{communate with derivative of third derivative}
	\begin{split}
	&I+II\\
=& 2 \sum_{j}F^{i\bar{i}}\textrm{Re}\Big\{(e_{i}\bar{e}_{i}e_{j}u-[e_{i},\bar{e}_{i}]^{0,1}e_{j}u)\bar{e}_{j}(u)\Big\} \\
	= & 2 \sum_{j}F^{i\bar{i}}\textrm{Re} \Big\{(e_{j}e_{i}\bar{e}_{i}u+e_{i}[\bar{e}_{i},e_{j}]u
	+[e_{i},e_{j}]\bar{e}_{i}u-[e_{i},\bar{e}_{i}]^{0,1}e_{j}u)\bar{e}_{j}(u)\Big\}\\
	= & 2\sum_{j}\textrm{Re}\{e_{j}(h)\bar{e}_{j}(u)\}-2 \sum_{j}F^{i\ol{i}}\textrm{Re}\{e_{j}(g_{i\ol{i}})\bar{e}_{j}(u)\}
+2\sum_{j}\textrm{Re}\{F^{i\bar{i}}e_{j}[e_{i},\bar{e}_{i}]^{0,1}u\bar{e}_{j}(u)\}
	\\
	&+2\textrm{Re}\sum_{j}\Big\{F^{i\bar{i}}\left(e_{i}[\bar{e}_{i},e_{j}]u
	+[e_{i},e_{j}]\bar{e}_{i}u-[e_{i},\bar{e}_{i}]^{0,1}e_{j}u\right)\bar{e}_{j}(u)\Big\}.\\
	\end{split}
\end{equation}
	We may assume $|\nabla u|>1$. It follows
	\begin{equation}\label{4.9}
	\begin{split}
	I+II\geq& 2\sum_{j}\textrm{Re}\{e_{j}(h)\bar{e}_{j}u\}-C|\nabla u|\sum_{j}F^{i\bar{i}}(|e_{i}e_{j}u|+|e_{i}\bar{e}_{j}u|)-C|\nabla u|^{2} \mathcal{F}\\
	\geq &2\sum_{j}\textrm{Re}\{e_{j}(h)\bar{e}_{j}u\}-\frac{C}{\varepsilon}|\nabla u|^{2}\mathcal{F}-\ve\sum_{j} F^{i\bar{i}}(|e_{i}e_{j}u|^{2}+|e_{i}\bar{e}_{j}u|^{2}).
	\end{split}
	\end{equation}
	Combining \eqref{4.9} with \eqref{calculation gradient}, we obtain
	\begin{equation}
	L(|\nabla u|^{2})\geq 2\sum_{j}\textrm{Re}\{e_{j}(h)\bar{e}_{j}u\}+ (1-\varepsilon) \sum_{j}F^{i\bar{i}}(|e_{i}e_{j}u|^{2}+|e_{i}\bar{e}_{j}u|^{2})- \frac{C}{\varepsilon} \mathcal{F}|\nabla u|^{2}.
	\end{equation}
\end{proof}
Using the above Lemma, it follows
\begin{equation}\label{3.7}
\begin{split}
\frac{L(|\nabla u|^{2})}{B_{1}\zeta|\nabla u|^{2}}\geq& \frac{2}{B_1\zeta|\nabla u|^{2}}\sum_{j}\textrm{Re}\{e_{j}(h)\bar{e}_{j}u\}\\&+ (1-\varepsilon) \sum_{j}F^{i\bar{i}}\frac{|e_{i}e_{j}u|^{2}+|e_{i}\bar{e}_{j}u|^{2}}{B_1\zeta|\nabla u|^{2}}- \frac{C}{B_1\zeta\varepsilon} \mathcal{F}.
\end{split}
\end{equation}
Now we estimate the last term of (\ref{3.9}). By the Cauchy-Schwarz inequality, for $0<\ve\leq \frac{1}{2}$, we have
\[
\begin{split}
&2\sum_{j}F^{i\bar{i}}\textrm{Re}\Big\{e_{i}(\eta)
\bar{e}_{i}e_{j}(u)\bar{e}_{j}(u)\Big\}\\
=& 2\sum_{j}F^{i\bar{i}}\textrm{Re}\Big\{e_{i}(\eta)
\bar{e}_{j}(u)\big\{e_{j}\bar{e}_{i}(u)-[e_{j},\bar{e}_{i}]^{0,1}(u)
-[e_{j},\bar{e}_{i}]^{1,0}(u)\big\}\Big\}\\
=&  2F^{i\bar{i}}\lambda_{i}\textrm{Re}\big\{e_{i}(\eta)
\bar{e}_{i}(u)\big\}
-2\sum_{j}F^{i\bar{i}}\textrm{Re}\left\{e_{i}(\eta)
\bar{e}_{j}(u) g_{i\bar{j}}\right\}\\
&	-2\sum_{j}F^{i\bar{i}}\textrm{Re}\big\{e_{i}(\eta)
\bar{e}_{j}(u)[e_{j},\bar{e}_{i}]^{1,0}(u)\big\}\\
\geq&  2F^{i\bar{i}}\lambda_{i}\textrm{Re}\{e_{i}(\eta)
\bar{e}_{i}(u)\}-\varepsilon B_1\zeta|\nabla u|^{2}F^{i\bar{i}}|e_{i}(\eta)|^{2}
-\frac{C}{B_1\zeta\varepsilon}|\nabla u|^{2}\mathcal{F},
\end{split}
\]
where in the last inequality we used $|\nabla u|>1.$
When $0<\varepsilon\leq\frac{1}{2}$, we have $1\leq (1-\varepsilon)(1+2\varepsilon)$. Using the Cauchy-Schwarz inequality again we obtain
\[
\begin{split}
&2\sum_{j}F^{i\bar{i}}\textrm{Re}\Big\{e_{i}(\eta)\bar{e}_{i}\bar{e}_{j}(u)e_{j}(u)\Big\} \\
\geq & -\frac{(1-\varepsilon)}{B_1\zeta}\sum_{j}F^{i\bar{i}}|\bar{e}_{i}\bar{e}_{j}(u)|^{2}
-(1+2\varepsilon)B_1\zeta|\nabla u|^{2}F^{i\bar{i}}|e_{i}(\eta)|^{2}.
\end{split}
\]
Therefore,
\begin{equation}\label{3.12re}
\begin{split}
&\frac{2}{|\nabla u|^{2}}\sum_{j}F^{i\bar{i}}\textrm{Re}\Big\{e_{i}(\eta)\bar{e}_{i}e_{j}(u)\bar{e}_{j}(u)
+e_{i}(\eta)\bar{e}_{i}\bar{e}_{j}(u)e_{j}(u)\Big\}\\
\geq &2F^{i\bar{i}}\lambda_{i}\frac{\textrm{Re}\{e_{i}(\eta)
	\bar{e}_{i}(u)\}}{|\nabla u|^{2}}-(1+3\ve)B_1\zeta F^{i\bar{i}}|e_{i}(\eta)|^{2}
-\frac{C}{B_1\zeta\varepsilon}\mathcal{F}
-(1-\varepsilon)\sum_{j}F^{i\bar{i}}\frac{|\bar{e}_{i}\bar{e}_{j}(u)|^{2}}{B_1\zeta|\nabla u|^{2}}.
\end{split}
\end{equation}
Then, using  \eqref{3.9},  \eqref{3.7} and \eqref{3.12re}, we obtain
\begin{equation}\label{(3).3.7}
\begin{split}
0 \geq & L(\eta)+B_1(1+\zeta)F^{i\bar{i}}|e_{i}(\eta)|^{2}-\frac{2C}{B_1\zeta\varepsilon}\mathcal{F}
+\frac{2}{B_1\zeta|\nabla u|^{2}}\sum_{j}\textrm{Re}\{e_{j}(h)\bar{e}_{j}u\}\\
&+2F^{i\bar{i}}\lambda_{i}\frac{\textrm{Re}\{e_{i}(\eta)
	\bar{e}_{i}(u)\}}{|\nabla u|^{2}}-(1+3\ve)B_1\zeta F^{i\bar{i}}{|e_{i}(\eta)|^{2}}\\
\geq &  L(\eta)+B_1(1-3\ve \zeta) F^{i\bar{i}}|e_{i}(\eta)|^{2}-\frac{2C}{B_1\zeta\varepsilon}\mathcal{F}
-\frac{C}{B_1\zeta|\nabla u|}
+2F^{i\bar{i}}\lambda_{i}\frac{\textrm{Re}\{e_{i}(\eta)
	\bar{e}_{i}(u)\}}{|\nabla u|^{2}}.
\end{split}
\end{equation}
We have $\ve=\frac{1}{12\sup_{x\in M}\zeta(x)}\leq \frac{1}{2}$ if $A$ is big enough. It follows
\begin{equation}
B_1(1-3\ve \zeta) F^{i\bar{i}}|e_{i}(\eta)|^{2}\geq \frac{1}{2}B_1F^{i\bar{i}}|e_{i}(\eta)|^{2}.
\end{equation}
We use the Cauchy-Schwarz inequality to obtain
\begin{equation}\label{(3).3.9}
\begin{split}
F^{i\bar{i}}\lambda_{i}\frac{2\textrm{Re}\{e_{i}(u)\bar{e_{i}}(\eta)\}}{|\nabla u|^{2}} \geq    -\frac{B_{1}}{4}F^{i\bar{i}}|e_{i}(\eta)|^{2}
-\frac{4}{B_1|\nabla u|^{2}}F^{i\bar{i}}\lambda_{i}^{2}.
\end{split}
\end{equation}
Combining with \eqref{(3).3.7}-\eqref{(3).3.9}, we have

\begin{equation}\label{last inequality of firt}
\begin{split}
\frac{B_1}{4}F^{i\bar{i}}|e_{i}(\eta)|^{2}+L(\eta)&\leq \frac{C}{B_1\zeta|\nabla u|}+ \frac{C}{B_1\zeta}\mathcal{F}+\frac{4}{B_1|\nabla u|^{2}}F^{i\bar{i}}\lambda_{i}^{2}\\
&\leq \frac{C}{B_1\zeta|\nabla u|}+ \frac{C}{B_1\zeta}\mathcal{F}+\frac{4n}{B_1|\nabla u|^{2}},
\end{split}
\end{equation}
where in the last inequality we used $F^{i\bar{i}}\lambda_{i}^{2}=\frac{\lambda_{i}^{2}}{1+\lambda_{i}^{2}}<1$ provided by \eqref{the first order derivative of F} for each $i=1,\cdots,n$.

The proof is divided to two cases, where $\lambda=(\lambda_{1},\cdots, \lambda_{n})$.

\textbf{Case (a)}. First, suppose \eqref{2..25} holds, i.e.
\begin{equation}\label{CS}
L(\eta)\geq \theta\mathcal{F}.
\end{equation}
Therefore, using \eqref{sub3} and \eqref{last inequality of firt}, we have
\[\frac{1}{2}\theta\mathcal{K}
+\frac{\theta}{2}\mathcal{F}\leq \theta\mathcal{F}\leq  \frac{C}{B_1\zeta|\nabla u|}+ \frac{C}{B_1\zeta}\mathcal{F}+\frac{4n}{B_1|\nabla u|^{2}}. \]
Note that the terms involving $\mathcal{F}$ can be discarded if $B_1$ big enough. Then we have
\[
\frac{1}{2}\theta\mathcal{K}\leq \frac{C}{B_1\zeta|\nabla u|}+\frac{4}{B_1|\nabla u|^{2}}.
\]
It follows $|\nabla u|\leq C.$

\textbf{Case (b)}. Second, suppose \eqref{2..26} holds. Then, by \eqref{sub3}, we have
\begin{equation}\label{4...15}
1\geq F^{i\bar{i}}\geq \theta\mathcal{F}\geq \theta\mathcal{K}, ~ \textrm{for}~i=1,2,\cdots, n.
\end{equation}
Hence, by \eqref{the first order derivative of F}, we have $|\lambda_{i}|\leq C$ and $F^{i\ol{i}}\leq1.$
It follows
\begin{equation}\label{4.31}
L(\eta)=F^{i\bar{i}}(g_{i\bar{i}}+\underline{u}_{i\bar{i}}-\lambda_{i})\geq -C.
\end{equation}
By \eqref{4...15}, \eqref{4.31}  and \eqref{last inequality of firt}, we obtain
\begin{equation*}
-C+\frac{1}{C}|\nabla \eta|^{2}\leq \frac{C}{B_1\zeta|\nabla u|}+\frac{4n}{B_1|\nabla u|^{2}}+C.
\end{equation*}
We can assume that $|\nabla u|\geq 2|\nabla \underline{u}|$. Hence $|\nabla \eta|\geq \frac{1}{2}|\nabla u|$.
It follows
\begin{equation*}
-C+\frac{1}{C}|\nabla u|^{2}\leq \frac{C}{B_1\zeta|\nabla u|}+\frac{4n}{B_1|\nabla u|^{2}}+C.
\end{equation*}
Therefore,  $|\nabla u|\leq C.$


\section{Second order estimates}
In this section, we prove the following second order estimates.
\begin{theorem}\label{Thm4.1}
There exists a constant $C_{0}>0$ such that
\begin{equation}\label{aa}
\|\nabla^{2}u\|_{C^{0}(M)}\leq C_{0},
\end{equation}
where $C_{0}$ depends on $(M, \chi, J)$, $\omega$, $\|h\|_{C^{2}}$, $\inf_{M} h$ and  $\underline{u}$ and $\nabla$ is the Levi-Civita connection of $\chi$.
\end{theorem}
~\\
Let $\mu_{1}(\nabla^{2} u)\geq \cdots\geq\mu_{2n}(\nabla^{2}u)$ be the  eigenvalues of $\nabla^{2}u$ with respect to $\chi.$
By \eqref{psh}, we have
$\sum_{\beta=1}^{2n}\mu_{\beta}=\Delta u=\Delta^{\mathbb{C}}u+\tau(du)=\sum_{i}\lambda_{i}+\tau(du)\geq \tau(du)\geq -C,$
(see \cite[(2.5)]{CTW16}).
Then $\mu_{2n}\geq -C\mu_{1}-C,$ which implies
\begin{equation*}
|\nabla^{2}u|_{g}\leq C\mu_{1}(\nabla^{2}u)+C,
\end{equation*}
for a uniform constant $C$. 
 Hence, it suffices to give an upper bound for $\mu_{1}$. First, we consider the function
\[
\tilde{Q}:= \log \mu_{1}(\nabla^{2}u)+\phi(|\nabla u|^{2})+\varphi(\widetilde{\eta})
\]
on $\Omega:=\{\mu_{1}(\nabla^{2}u)>0\}\subset M$. Here $\varphi$ is a function defined by
\[
\varphi(\widetilde{\eta}):= e^{B\widetilde{\eta}},
~~\widetilde{\eta}:=\underline{u}-u+\sup_{M}(u-\underline{u})+1
\]
for a real constant $B>0$ to be determined later, and $\phi$ is defined by
\[
\phi(s):=-\frac{1}{2}\log(1+\sup_{M}|\nabla u|^{2}-s).
\]
Set $K=1+\sup_{M}|\nabla u|^{2}$. Note that
\begin{equation}\label{property test function}
   \frac{1}{2K}\leq \phi'(|\nabla u|^{2})\leq \frac{1}{2},~~\phi''=2(\phi')^{2}.
\end{equation}
We may assume $\Omega$ is a nonempty open set (otherwise we are done). Note that when $z$ approaches to $\partial{\Omega}$, then $Q(z)\ri -\infty$.
Suppose $Q$ achieves a maximum at $x_{0}$ in $\Omega$.
Near $x_{0}$, choose a local unitary frame $\{e_{i}\}_{i=1}^{n}$ (with respect to $\chi$) such that at $x_{0}$,
\begin{equation}\label{5..6}
\text{$\chi_{i\overline{j}}=\delta_{ij}$, $\tilde{g}_{i\overline{j}}=\delta_{ij}\tilde{g}_{i\overline{j}}$ and $\tilde{g}_{1\overline{1}}\geq\tilde{g}_{2\overline{2}}\geq\cdots\geq\tilde{g}_{n\overline{n}}$.}
\end{equation}
For convenience, we denote $\tilde{g}_{i\overline{i}}(x_{0})$ by $\lambda_{i}$. In addition, there exists a normal coordinate system $(U,\{x^{\alpha}\}_{i=1}^{2n})$ in a neighbourhood of $x_{0}$ such that
\begin{equation}\label{real frame and complex frame}
e_{i}=\frac{1}{\sqrt{2}}(\partial_{2i-1}-\sqrt{-1}\partial_{2i}) \text{~for~} i=1,2,\cdots,n
\end{equation}
and
\begin{equation}\label{first derivative of background metric}
\frac{\partial \chi_{\alpha\beta}}{\partial x^{\gamma}}=0 \text{~for~} \alpha,\beta,\gamma=1,2,\cdots,2n
\end{equation}
at $x_{0}$, where $\chi_{\alpha\beta}=\chi(\partial_{\alpha}, \partial_{\beta}).$

Suppose that $V_{1},\cdots, V_{2n}$ are $\chi$-unit  eigenvectors for $\Phi$ at $x_{0}$ with eigenvalues $$\mu_{1}(\nabla^{2}u)\geq \cdots\geq\mu_{2n}(\nabla^{2} u),$$ respectively.

Assume $V_{\alpha}=V_{\alpha}^{\beta}\partial_{\beta}$ at $x_{0}$ and extend vectors $V_{\alpha}$ to  vector fields  on $U$ by taking the components $V_{\alpha}^{\beta}$ to be constant.
Since $\mu_{1}(\nabla^{2}u)$ may not be smooth, we apply a perturbation argument as in  \cite{CTW16,SG}.  On $U$, define
\begin{equation}\label{Definition of Phi}
\begin{split}
\Phi & = \Phi_{\alpha}^{\beta}~\frac{\partial}{\partial x^{\alpha}}\otimes dx^{\beta}\\
& = (g^{\alpha\gamma}u_{\gamma\beta}-g^{\alpha\gamma}B_{\gamma\beta})\frac{\partial}{\partial x^{\alpha}}\otimes dx^{\beta},
\end{split}
\end{equation}
where $B_{\gamma\beta}=\delta_{\gamma\beta}-V_{1}^{\gamma}V_{1}^{\beta}$.  Assume that  $\mu_{1}(\Phi)\geq\mu_{2}(\Phi)\geq\cdots\geq\mu_{2n}(\Phi)$  are  the eigenvalues of $\Phi$.  Then  $V_{1},V_{2},\cdots,V_{2n}$ are still eigenvectors of $\Phi$, corresponding to eigenvalues $\mu_{1}(\Phi),\mu_{2}(\Phi),\cdots,\mu_{2n}(\Phi)$ at $x_{0}$.   Note that $\mu_{1}(\Phi)(x_{0})>\mu_{2}(\Phi)(x_{0})$, which implies $\mu_{1}(\Phi)$ is smooth near $x_{0}$.   On $U$, we replace $\tilde{Q}$ by  the following smooth quantity
\begin{equation*}
Q=\log\mu_{1}(\Phi)+\phi(|\nabla u|^{2})+\varphi(\widetilde{\eta}).
\end{equation*}
Since $\mu_{1}(\nabla^{2}u)(x_{0})=\mu_{1}(\Phi)(x_{0})$ and $\mu_{1}(\nabla^{2} u)\geq\mu_{1}(\Phi)$, $x_{0}$ is still the maximum point of $\hat{Q}$. For convenience,  we denote $\mu_{\alpha}(\Phi)$ by $\mu_{\alpha}$ for $\alpha=1,2,\cdots,2n$.


The proof needs the first and second derivatives of the first eigenvalue $\mu_{1}$ at $x_{0}$ (See \cite[Lemma 5.2]{CTW16} or \cite{Spr,SG}).
\begin{lemma}\label{the derivative of eigenvalues}
	At $x_{0}$, we have
	\begin{equation}
	\begin{split}
	                              \frac{\partial \mu_{1}}{\partial \Phi^{\alpha}_{\beta} }&=V_{1}^{\alpha}V_{1}^{\beta};\\
\frac{\partial^{2} \mu_{1}}{\partial \Phi^{\alpha}_{\beta}\partial \Phi^{\gamma}_{\delta}}&
=\sum_{\kappa>1}\frac{1}{\mu_{1}-\mu_{\kappa}}(V_{1}^{\alpha}V_{\kappa}^{\beta}V_{\kappa}^{\gamma}V_{1}^{\delta}
	+V_{\kappa}^{\alpha}V_{1}^{\beta}V_{1}^{\gamma}V_{\kappa}^{\delta}).
	\end{split}
	\end{equation}
\end{lemma}


 At $x_{0}$, for each $i=1,2, \cdots, n$, we have
\begin{equation}\label{5..10}
\frac{1}{\mu_{1}}e_{i}(\mu_{1})=-\phi'e_{i}(|\nabla u|^{2})-Be^{B\widetilde{\eta}}e_{i}(\widetilde{\eta})
\end{equation}
and
\begin{equation}\label{(2)3.4}
\begin{split}
0\geq L(Q)=& \frac{L(\mu_{1})}{\mu_{1}}- F^{i\bar{i}}\frac{|e_{i}(\mu_{1})|^{2}}{\mu_{1}^{2}}  +\phi'' F^{i\bar{i}}|e_{i}(|\nabla u|^{2})|^{2} \\
      &+\phi'L(|\nabla u|^{2})+Be^{B\widetilde{\eta}}L(\widetilde{\eta})+B^{2}e^{B\widetilde{\eta}} F^{i\bar{i}}|e_{i}(\widetilde{\eta})|^{2}.
\end{split}
\end{equation}
\subsection{Lower bound for $L(Q)$}
In this subsection, we calculation  $L(Q)$.
\begin{lemma}\label{Lma4.2}
For $\ve\in (0,\frac{1}{2}]$, at $x_{0}$, we have
\begin{equation}\label{4.7"}
\begin{split}
L(Q) \geq & (2-\ve)\sum_{\beta>1}F^{i\bar{i}}\frac{|e_{i}(u_{V_{\beta}V_{1}})|^{2}}{\mu_{1}(\mu_{1}-\mu_{\beta})}
-\frac{1}{\mu_{1}} F^{i\bar{k},j\bar{l}}V_{1}(\tilde{g}_{i\bar{k}})V_{1}(\tilde{g}_{j\bar{l}})\\&
-(1+\ve)F^{i\bar{i}}\frac{|e_{i}(\mu_{1})|^{2}}{\mu_{1}^{2}}
-\frac{C}{\ve}\mathcal{F}
+ \frac{\phi'}{2}\sum_{j} F^{i\bar{i}}(|e_{i}e_{j}u|^{2}+|e_{i}\bar{e}_{j}u|^{2})\\&+\phi'' F^{i\bar{i}}|e_{i}(|\nabla u|^{2})|^{2}+Be^{B\widetilde{\eta}}L(\widetilde{\eta})+B^{2}e^{B\widetilde{\eta}} F^{i\bar{i}}|e_{i}(\widetilde{\eta})|^{2}.\\
\end{split}
\end{equation}
\end{lemma}
 First, we calculate $L(\mu_{1})$. Let $u_{ij}=e_{i}e_{j}u-(\nabla_{e_{i}}e_{j})u $ and $u_{V_{\alpha}V_{\beta}}=u_{\gamma\delta}V^{\gamma}_{\alpha}V^{\delta}_{\beta}$.

 By Lemma \ref{the derivative of eigenvalues} and \eqref{first derivative of background metric}, we have
\begin{equation}\label{5..13}
\begin{split}
L(\mu_{1})=&F^{i\bar{i}}\frac{\partial^{2} \mu_{1}}{\partial \Phi^{\alpha}_{\beta}\partial \Phi^{\gamma}_{\delta}}
e_{i}(\Phi^{\gamma}_{\delta})\bar{e}_{i}(\Phi^{\alpha}_{\beta})+F^{i\bar{i}}\frac{\partial \mu_{1}}{\partial
	\Phi^{\alpha}_{\beta}}(e_{i}\bar{e}_{i}-[e_{i},\bar{e}_{i}]^{0,1})(\Phi^{\alpha}_{\beta})\\
=&F^{i\bar{i}}\frac{\partial^{2} \mu_{1}}{\partial \Phi^{\alpha}_{\beta}\partial \Phi^{\gamma}_{\delta}}
e_{i}(u_{\gamma\delta})\bar{e}_{i}(u_{\alpha\beta})+F^{i\bar{i}}\frac{\partial \mu_{1}}{\partial
	\Phi^{\alpha}_{\beta}}(e_{i}\bar{e}_{i}-[e_{i},\bar{e}_{i}]^{0,1})(u_{\alpha\beta})\\
&+F^{i\bar{i}}\frac{\partial \mu_{1}}{\partial
	\Phi^{\alpha}_{\beta}}u_{\gamma\beta}e_{i}\bar{e}_{i}(\chi^{\alpha\gamma})\\
\geq & 2\sum_{\beta>1}F^{i\bar{i}}\frac{|e_{i}(u_{V_{\beta}V_{1}})|^{2}}{\mu_{1}-\mu_{\beta}}
+ F^{i\bar{i}}(e_{i}\bar{e}_{i}-[e_{i},\bar{e}_{i}]^{0,1})(u_{V_{1}V_{1}})-C\mu_{1}\mathcal{F},\\
\end{split}
\end{equation}
where $(\chi^{\alpha\beta})$ is the inverse  of the  matrix $(\chi_{\alpha\beta})$. Let $W$ be a vector field.
Differentiating the equation \eqref{DHYM}, we obtain
\begin{equation}\label{the first detivative of equation}
F^{i\bar{i}}W(\tilde{g}_{i\bar{i}})=W(h)
\end{equation}
and
\begin{equation}\label{5..15}
F^{i\bar{i}}V_{1}V_{1}(\tilde{g}_{i\bar{i}})=- F^{i\bar{k},j\bar{l}}V_{1}(\tilde{g}_{i\bar{k}})V_{1}(\tilde{g}_{j\bar{l}})+V_{1}V_{1}(h).
\end{equation}
Commuting the derivatives and using Proposition \ref{estimates of gradient}, we obtain, for any vector field $W$,
\begin{equation}\label{acting the first derivative}
|L(W(u))|\leq C+C\mu_{1}\mathcal{F}.
\end{equation}
~\\
\textbf{Claim 1.} If $\mu_{1}\gg 1$, then
\begin{equation}\label{claim2}
\begin{split}
 F^{i\bar{i}}(e_{i}\bar{e}_{i}-[e_{i},\bar{e}_{i}]^{0,1})(u_{V_{1}V_{1}})
 \geq   & - F^{i\bar{k},j\bar{l}}V_{1}(\tilde{g}_{i\bar{k}})V_{1}(\tilde{g}_{j\bar{l}})-C\mu_{1}\mathcal{F}\\&-2 F^{i\bar{i}}\Big\{[V_{1},\bar{e}_{i}]V_{1}e_{i}(u)+[V_{1},e_{i}]V_{1}\bar{e}_{i}(u)\Big\}. \\
\end{split}
\end{equation}
\textit{Proof.}
By \eqref{acting the first derivative}, we have
\begin{equation}\label{communate fourth derivative}
\begin{split}
    & F^{i\bar{i}}(e_{i}\bar{e}_{i}-[e_{i},\bar{e}_{i}]^{0,1})(u_{V_{1}V_{1}})\\
   =& F^{i\bar{i}}e_{i}\bar{e}_{i}(V_{1}V_{1}(u)-(\nabla_{V_{1}}V_{1})u)
   -F^{i\bar{i}}[e_{i},\bar{e}_{i}]^{0,1}(V_{1}V_{1}(u)-(\nabla_{V_{1}}V_{1})u)\\
   \geq & F^{i\bar{i}}e_{i}\bar{e}_{i}(V_{1}V_{1}(u))-F^{i\bar{i}}[e_{i},\bar{e}_{i}]^{0,1}V_{1}V_{1}(u)-C\mu_{1}\mathcal{F}-C.
 \end{split}
\end{equation}
 Recall the definition of Lie bracket $[e_{i}, e_{j}]=e_{i}e_{j}-e_{j}e_{i}.$
Then  we get (for more details, see \cite[(5.12)]{CTW16})
 \begin{equation*}
 \begin{split}
 & F^{i\overline{i}}e_{i}\overline{e}_{i}V_{1}V_{1}(\varphi)-F^{i\overline{i}}[e_{i},\overline{e}_{i}]^{(0,1)}V_{1}V_{1}(\varphi) \\
 \geq ~~& F^{i\overline{i}}\left(V_{1}e_{i}\overline{e}_{i}V_{1}(\vp)+[e_{i},V_{1}]\overline{e}_{i}V_{1}(\vp)
 -[V_{1},\overline{e}_{i}]e_{i}V_{1}(\vp)-V_{1}V_{1}[e_{i},\overline{e}_{i}]^{(0,1)}(\vp)\right)\\
 & -C\mu_{1}\mathcal{F}\\
 \geq ~~&F^{i\overline{i}}V_{1}V_{1}\left(e_{i}\overline{e}_{i}(\varphi)-[e_{i},\overline{e}_{i}]^{(0,1)}(\varphi)\right)
 -2F^{i\overline{i}}[V_{1},e_{i}]V_{1}\overline{e}_{i}(\vp)\\
 & -2F^{i\overline{i}}[V_{1},\overline{e}_{i}]V_{1}e_{i}(\vp)-C\mu_{1}\mathcal{F}.
 \end{split}
 \end{equation*}
Combining with \eqref{2..27}, \eqref{5..15} and \eqref{communate fourth derivative}, the Claim 1 follows if $\mu_{1}\gg 1$.

\qed
~\\

Combining the equalities (\ref{5..13}) and (\ref{claim2}) together, it follows that
\begin{equation}\label{(2)3.9}
\begin{split}
L(\mu_{1})\geq & 2\sum_{\beta>1}F^{i\bar{i}}\frac{|e_{i}(u_{V_{\beta}V_{1}})|^{2}}{\mu_{1}-\mu_{\beta}}
-F^{i\bar{k},j\bar{l}}V_{1}(\tilde{g}_{i\bar{k}})V_{1}(\tilde{g}_{j\bar{l}})
\\&-2F^{i\bar{i}}\textrm{Re}\Big\{[V_{1},e_{i}]V\bar{e}_{i}(u)
+[V_{1},\bar{e}_{i}]Ve_{i}(u)\Big\}-C\mu_{1}\mathcal{F}.
\end{split}
\end{equation}
By \eqref{2..27}, Lemma \ref{calculation  of gradient} and Proposition \ref{estimates of gradient}, we have
\begin{equation}\label{(2)3.10}
L(|\nabla u|^{2})\geq  \frac{1}{2}\sum_{j} F^{i\bar{i}}(|e_{i}e_{j}u|^{2}+|e_{i}\bar{e}_{j}u|^{2})- C\mathcal{F}.
\end{equation}
Substituting \eqref{(2)3.9} and \eqref{(2)3.10} into \eqref{(2)3.4}, we obtain
\begin{equation}\label{4.10-1}
\begin{split}
L(Q)
\geq & 2\sum_{\beta>1}F^{i\bar{i}}\frac{|e_{i}(u_{V_{\beta}V_{1}})|^{2}}{\mu_{1}(\mu_{1}-\mu_{\beta})}
-\frac{1}{\mu_{1}} F^{i\bar{k},j\bar{l}}V_{1}(\tilde{g}_{i\bar{k}})V_{1}(\tilde{g}_{j\bar{l}})+B^{2}e^{B\widetilde{\eta}} F^{i\bar{i}}|e_{i}(\widetilde{\eta})|^{2}
\\&+Be^{B\widetilde{\eta}}L(\widetilde{\eta})-2 F^{i\bar{i}}\frac{\textrm{Re}\{[V_{1},e_{i}]V_{1}\bar{e}_{i}(u)+[V_{1},\bar{e}_{i}]V_{1}e_{i}(u)\}}{\mu_{1}}
-C\mathcal{F}\\
&-F^{i\bar{i}}\frac{|e_{i}(\mu_{1})|^{2}}{\mu_{1}^{2}}
+ \frac{\phi'}{2}\sum_{j} F^{i\bar{i}}(|e_{i}e_{j}u|^{2}+|e_{i}\bar{e}_{j}u|^{2})+\phi'' F^{i\bar{i}}|e_{i}(|\nabla u|^{2})|^{2}.\\
\end{split}
\end{equation}
Now we deal with the third order derivatives of the right hand side of \eqref{4.10-1}.
~\\
\textbf{Claim 2.} For any $\ve\in (0,\frac{1}{2}]$, we have
\begin{equation}\label{4.10}
\begin{split}
   & 2 F^{i\bar{i}}\frac{\textrm{Re}\{[V_{1},e_{i}]V_{1}\bar{e}_{i}(u)+[V_{1},\bar{e}_{i}]V_{1}e_{i}(u)\}}{\mu_{1}}\\
    \leq & \ve F^{i\bar{i}}\frac{|e_{i}(u_{V_{1}V_{1}})|^{2}}{\mu_{1}^{2}}+
    \ve\sum_{\beta>1}F^{i\bar{i}}\frac{|e_{i}(u_{V_{\beta}V_{1}})|^{2}}{\mu_{1}(\mu_{1}-\mu_{\beta})}
    +\frac{C}{\ve}\mathcal{F}.
\end{split}
\end{equation}
\begin{proof}
Asume 
\[
[V_{1},e_{i}]=\sum_{\beta=1}^{2n} \mu_{i\beta}V_{\beta},~[V_{1},\bar{e}_{i}]=\sum_{\beta=1}^{2n} \overline{\mu_{i\beta}}V_{\beta},
\]
where $\mu_{i\beta}\in\mathbb{C}$ are constants.
Thus,
\begin{equation}\label{three dv}
\textrm{Re}\{[V_{1},e_{i}]V_{1}\bar{e}_{i}(u)+[V_{1},\bar{e}_{i}]V_{1}e_{i}(u)\}\leq C\sum_{\beta=1}^{2n}|V_{\beta}V_{1}e_{i}(u)|.
\end{equation}
Then we are reduced to estimating $\underset{\beta}{\sum} F^{i\bar{i}}\frac{|V_{\beta}V_{1}e_{i}(u)|}{\mu_{1}}$.
Using the definition of Lie bracket $e_{i}e_{j}-e_{j}e_{i}=[e_{i},e_{j}]$, we have
\[
\begin{split}
  \big|V_{\beta}V_{1}e_{i}(u)\big|= &\big|e_{i}V_{\beta}V_{1}(u)+V_{\beta}[V_{1},e_{i}](u)+[V_{\beta},e_{i}]V_{1}(u)\big|  \\
    =& \big|e_{i}(u_{V_{\beta}V_{1}})+e_{i}(\nabla_{V_{\beta}}V_{1})(u)+V_{\beta}[V_{1},e_{i}](u)
    +[V_{\beta},e_{i}]V_{1}(u)\big|\\
    \leq &\big|e_{i}(u_{V_{\beta}V_{1}})\big|+C\mu_{1}.
\end{split}
\]
Therefore,
\begin{equation}\label{third order derivative 1}
\begin{split}
    \sum_{\beta} F^{i\bar{i}}\frac{|V_{\beta}V_{1}e_{i}(u)|}{\mu_{1}}
    \leq &  \sum_{\beta} F^{i\bar{i}}\frac{|e_{i}(u_{V_{\beta}V_{1}})|}{\mu_{1}}+C\mathcal{F}\\
     = & F^{i\bar{i}}\frac{|e_{i}(u_{V_{1}V_{1}})|}{\mu_{1}}+\sum_{\beta>1} F^{i\bar{i}}\frac{|e_{i}(u_{V_{\beta}V_{1}})|}{\mu_{1}}+C\mathcal{F}.\\
\end{split}
\end{equation}
By the Cauchy-Schwarz inequality, for $\ve\in (0,\frac{1}{2}]$, we derive
\[
F^{i\bar{i}}\frac{|e_{i}(u_{V_{1}V_{1}})|}{\mu_{1}}\leq \ve F^{i\bar{i}}\frac{|e_{i}(u_{V_{1}V_{1}})|^{2}}{\mu_{1}^{2}}+\frac{C}{\ve}\mathcal{F}
\]
and
\[
\begin{split}
\sum_{\beta>1} F^{i\bar{i}}\frac{|e_{i}(u_{V_{\beta}V_{1}})|}{\mu_{1}}\leq & \ve F^{i\bar{i}}\frac{|e_{i}(u_{V_{\beta}V_{1}})|^{2}}{\mu_{1}(\mu_{1}-\mu_{\beta})}
+\sum_{\beta>1}\frac{\mu_{1}-\mu_{\beta}}{\ve\mu_{1}}\mathcal{F} \\
    \leq & \ve\sum_{\beta>1} F^{i\bar{i}}\frac{|e_{i}(u_{V_{\beta}V_{1}})|^{2}}{\mu_{1}(\mu_{1}-\mu_{\beta})}
    +\frac{C}{\ve}\mathcal{F},
\end{split}
\]
where the last inequality we used
$\sum_{\beta=1}^{2n}\mu_{\beta}=\Delta u=\Delta^{\mathbb{C}}u+\tau(du)\geq -C+\tau(du)\geq -C$ (see \cite[(2.5)]{CTW16}).
Here  $\tau$ is the torsion vector field of $(\chi, J)$ (the dual of its Lee form, see
e.g. \cite[Lemma 3.2]{TV07}).
Combining with the above three inequalities, we have
\[
\sum_{\beta} F^{i\bar{i}}\frac{|V_{\beta}V_{1}e_{i}(u)|}{\mu_{1}}\leq
\ve F^{i\bar{i}}\frac{|e_{i}(u_{V_{1}V_{1}})|^{2}}{\mu_{1}^{2}}+
\ve\sum_{\beta>1} F^{i\bar{i}}\frac{|e_{i}(u_{V_{\beta}V_{1}})|^{2}}{\mu_{1}(\mu_{1}-\mu_{\beta})}
+\frac{C}{\ve}\mathcal{F}.
\]
Then by \eqref{three dv}, it follows (\ref{4.10}).
\end{proof}
Consequently, Lemma \ref{Lma4.2} follows from (\ref{4.10-1}) and (\ref{4.10}). Now we continue to prove Theorem \ref{Thm4.1}.
\subsection{Proof of Theorem \ref{Thm4.1}} The proof can be divided into three cases.
\textbf{Case 1:} \begin{equation}\label{4.12}
F^{n\bar{n}}\leq B^{3}e^{2B\widetilde{\eta}(0)}F^{1\bar{1}}.\end{equation}

In this case, we can choose $\ve=\frac{1}{2}$.  Using the elemental inequality $|a+b|^{2}\leq 4|a|^{2}+\frac{4}{3}|b|^{2}$ for (\ref{5..10}), we get
\begin{equation}\label{third order 1}
-(1+\ve)F^{i\bar{i}}\frac{|e_{i}(\mu_{1})|^{2}}{\mu_{1}^{2}}\geq
-6\sup_{M}(|\nabla \widetilde{\eta}|^{2})B^{2}e^{2B\widetilde{\eta}}\mathcal{F}-2(\phi')^{2} F^{i\bar{i}}|e_{i}(|\nabla u|^{2})|^{2}.
\end{equation}
Plugging \eqref{third order 1} and \eqref{property test function} into (\ref{4.7"}), we get,
\begin{equation}\label{(2)3.17}
\begin{split}
L(Q)  \geq & (2-\ve)\sum_{\beta>1}F^{i\bar{i}}\frac{|e_{i}(u_{V_{\beta}V_{1}})|^{2}}{\mu_{1}(\mu_{1}-\mu_{\beta})}
-\frac{1}{\mu_{1}} F^{i\bar{k},j\bar{l}}V_{1}(\tilde{g}_{i\bar{k}})V_{1}(\tilde{g}_{j\bar{l}})
\\&-\Big(\frac{C}{\ve}+6\sup_{M}\{|\nabla \widetilde{\eta}|^{2}\}B^{2}e^{2B\widetilde{\eta}}\Big)\mathcal{F}
+ \frac{\phi'}{2}\sum_{j} F^{i\bar{i}}(|e_{i}e_{j}u|^{2}+|e_{i}\bar{e}_{j}u|^{2})\\&
+Be^{B\widetilde{\eta}}L(\widetilde{\eta})+B^{2}e^{B\widetilde{\eta}} F^{i\bar{i}}|e_{i}(\widetilde{\eta})|^{2}\\
\geq & -\Big(\frac{C}{\ve}+6\sup_{M}\{|\nabla \widetilde{\eta}|^{2}\}B^{2}e^{2B\widetilde{\eta}}\Big)\mathcal{F}
+ \frac{\phi'}{2}\sum_{j} F^{i\bar{i}}(|e_{i}e_{j}u|^{2}+|e_{i}\bar{e}_{j}u|^{2})\\&
+Be^{B\widetilde{\eta}}L(\widetilde{\eta}),
\end{split}
\end{equation}
where the last inequality we used the concavity of $F$.
Note $F^{i\bar{i}}\leq 1$. We have
\begin{equation}\label{lower bound of Leta}
\begin{split}
L(\widetilde{\eta})=&\sum_{i}F^{i\bar{i}}(\underline{u}_{i\bar{i}}-u_{i\bar{i}})
=\sum_{i}F^{i\bar{i}}(g_{i\bar{i}}+\underline{u}_{i\bar{i}}-\tilde{g}_{i\bar{i}})\\
\geq &-C-\sum_{i}F^{i\bar{i}}\tilde{g}_{i\bar{i}}
=-C-\sum_{i}\frac{\lambda_{i}}{1+\lambda_{i}^{2}}
\geq -C.
\end{split}
\end{equation}
Then, by \eqref{2..26}, we have
\begin{equation}\label{z}
0\geq L(Q)\geq \frac{\phi'}{2}\sum_{j} F^{i\bar{i}}(|e_{i}e_{j}u|^{2}+|e_{i}\bar{e}_{j}u|^{2})-C_{B}\mathcal{F}.\end{equation}
Here $C_{B}$ are positive constants depending on $B$ {which might different from line to line}.
By (\ref{4.12}), we have
\[
\sum_{i,j}(|e_{i}e_{j}u|^{2}+|e_{i}\bar{e}_{j}u|^{2})\leq C_{B}.
\]
Then the complex covariant derivatives
\[
u_{ij}=e_{i}e_{j}u-(\nabla_{e_{i}}e_{j})u, ~u_{i\bar{j}}=e_{i}\bar{e}_{j}u-(\nabla_{e_{i}}\bar{e}_{j})u
\]
satisfy
\[
\sum_{i,j}(|u_{ij}|^{2}+|u_{i\bar{j}}|^{2})\leq C_{B}.
\]


\textbf{Case 2:}
\begin{equation}\label{5.15}
\frac{\phi'}{4}\sum_{j} F^{i\bar{i}}(|e_{i}e_{j}u|^{2}+|e_{i}\bar{e}_{j}u|^{2})>6\sup_{M}(|\nabla \widetilde{\eta}|^{2})B^{2}e^{2B\widetilde{\eta}}\mathcal{F}.\end{equation}

Note that \eqref{(2)3.17} is still true. By \eqref{5.15} and \eqref{(2)3.17}, we  have
\begin{equation}\label{4.19}
\begin{split}
L(Q)  
 \geq & \frac{\phi'}{4}\sum_{j} F^{i\bar{i}}(|e_{i}e_{j}u|^{2}+|e_{i}\bar{e}_{j}u|^{2})
-\frac{C}{\ve}\mathcal{F}
+Be^{B\widetilde{\eta}}L(\widetilde{\eta}).
\end{split}
\end{equation}
~\\
(a). If (\ref{2..25}) holds, $L(\widetilde{\eta})\geq \theta\mathcal{F}$.  Then, by (\ref{4.19}), we have
\[
0 \geq \frac{\phi'}{4}\sum_{j} F^{i\bar{i}}(|e_{i}e_{j}u|^{2}+|e_{i}\bar{e}_{j}u|^{2})+\Big(\theta Be^{B\widetilde{\eta}}-\frac{C}{\ve}\Big)\mathcal{F}.
\]
This yields a contradiction if we further assume $B$ is large enough.
~\\
(b). If (\ref{2..26}) holds, then $F^{k\bar{k}}\geq \theta\mathcal{F}, k=1,2,\cdots,n.$
Then by (\ref{4.19}) and \eqref{lower bound of Leta} we obtain
\[
\begin{split}
  0\geq&\frac{\phi'}{4}\sum_{j} F^{i\bar{i}}(|e_{i}e_{j}u|^{2}+|e_{i}\bar{e}_{j}u|^{2})-C_{B}\mathcal{F} \\
    \geq & \frac{\phi'}{4}\theta\mathcal{F}\sum_{i,j} (|e_{i}e_{j}u|^{2}+|e_{i}\bar{e}_{j}u|^{2})-C_{B}\mathcal{F}.
\end{split}
\]
Therefore,
\[\sum_{i,j}(|e_{i}e_{j}u|^{2}+|e_{i}\bar{e}_{j}u|^{2})\leq C_{B}.\]
\bigskip

\textbf{Case 3:} If the Case 1  and Case 2 do not hold, define the index set
\begin{equation}\label{}
I:=\Big\{1\leq i\leq n:~ F^{n\bar{n}}\geq B^{3}e^{2B\widetilde{\eta}}F^{i\bar{i}}\Big\}.
\end{equation}
Clearly, we have $1\in I$ and $n\not\in I$. Hence we may write $I=\{1,2,\cdots, p\}$ for some positive integer $p<n$. Now we deal with the third order term.
\begin{lemma}\label{partial third order}
	Assume $B\geq 6n\sup_{M}|\nabla \widetilde{\eta}|^{2}$. At $x_{0}$, we have
\begin{equation}
-(1+\ve)\sum_{i\in I}F^{i\bar{i}} \frac{|e_{i}(\mu_{1})|^{2}}{\mu_{1}^{2}}\geq -\mathcal{F}-2(\phi')^{2}\sum_{i\in I}F^{i\bar{i}}|e_{i}(|\nabla u|^{2})|.
\end{equation}
\end{lemma}
\begin{proof}
Using (\ref{5..10}) and the inequality $|a+b|^{2}\leq 4|a|^{2}+\frac{4}{3}|b|^{2}$, we obtain
\[
\begin{split}
&-(1+\ve)\sum_{i\in I} F^{i\bar{i}}\frac{|e_{i}(\mu_{1})|^{2}}{\mu_{1}^{2}}\\
=&-(1+\ve)\sum_{i\in I} F^{i\bar{i}}|\phi'e_{i}(|\nabla u|^{2})+Be^{B\widetilde{\eta}}e_{i}(\widetilde{\eta})|^{2}\\
\geq & -6\sup_{M}|\nabla \widetilde{\eta}|^{2}B^{2}e^{2B\widetilde{\eta}}\sum_{i\in I} F^{i\bar{i}}-2(\phi')^{2}\sum_{i\in I} F^{i\bar{i}}|e_{i}(|\nabla u|^{2})|^{2}\\
\geq & -6n\sup_{M}|\nabla \widetilde{\eta}|^{2}B^{-1} F^{n\bar{n}}-2(\phi')^{2}\sum_{i\in I} F^{i\bar{i}}|e_{i}(|\nabla u|^{2})|^{2}\\
\geq &-\mathcal{F}-2(\phi')^{2}\sum_{i\in I}F^{i\bar{i}}|e_{i}(|\nabla u|^{2})|^{2},
\end{split}
\]
where we used the hypothesis $B\geq 6n\sup_{M}|\nabla \widetilde{\eta}|^{2}$ in the last second inequality.
\end{proof}

%
To deal with the bad third order terms,  we need to give a lower bound of the good third order terms from the concavity of the equation \eqref{DHYM}.

\begin{lemma}\label{lower bound of concave}
	 We have
	\begin{equation}
	\begin{split}
	&-\frac{1}{\mu_{1}} F^{i\bar{k},j\bar{l}}V_{1}(\tilde{g}_{i\bar{k}})V_{1}(\tilde{g}_{j\bar{l}})\\
	\geq &
	\frac{2}{\mu_{1}}\sum_{i\not\in I, k\in I}F^{i\bar{i}}\tilde{g}^{k\bar{k}}|V_{1}(\tilde{g}_{i\bar{k}})|^{2}
	+\frac{C_{B}}{\ve\mu_{1}^{4}}\sum_{i\not\in I,k\not\in I}F^{i\bar{i}}{|V_{1}(\tilde{g}_{i\bar{k}})|^{2}},\\
	\end{split}
	\end{equation}
where $(\tilde{g}^{k\bar{l}})$ is the inverse of the matrix $(\tilde{g}_{i\bar{j}}).$
\end{lemma}
\begin{proof}
	
	 By \eqref{the second order derivative of F}, we have
	\begin{equation}\label{third order term from concacity}
	\begin{split}
	& -F^{i\bar{k},j\bar{l}}V_{1}(\tilde{g}_{i\bar{k}})V_{1}(\tilde{g}_{j\bar{l}})  \\
	=& \sum_{i}\frac{2\lambda_{i}}{(1+\lambda_{i}^{2})^{2}}|V_{1}(\tilde{g}_{i\bar{i}})|^{2}
	+\sum_{i\neq k}\frac{\lambda_{i}+\lambda_{k}}{(1+\lambda_{i}^{2})(1+\lambda_{k}^{2})}|V_{1}(\tilde{g}_{i\bar{k}})|^{2}.
	\end{split}
	\end{equation}
Now we calculate these two terms.  We claim that
\begin{equation}\label{3..27}
\begin{split}
&\sum_{i\neq k}\frac{\lambda_{i}+\lambda_{k}}{(1+\lambda_{i}^{2})(1+\lambda_{k}^{2})}|V_{1}(\tilde{g}_{i\bar{k}})|^{2}\\
\geq &
2\sum_{i\not\in I, k\in I}\frac{1}{(1+\lambda_{i}^{2})\lambda_{k}}|V_{1}(\tilde{g}_{i\bar{k}})|^{2}
+\frac{C_{B}}{\ve\mu_{1}^{3}}\sum_{i\not\in I,k\not\in I,i\neq k}\frac{1}{1+\lambda_{i}^{2}}|V_{1}(\tilde{g}_{i\bar{k}})|^{2}.\\
\end{split}
\end{equation}
	
	Denote
	\begin{equation*}
	\begin{split}
	S_{1}:=&\{(i,k):\ i\not\in I,\ k\in I,\ 1\leq i,\ k\leq n\};\\[1.5mm]
	S_{2}:=&\{(i,k):\ i\neq k,\ i\not\in I,\ k\not\in I, \ 1\leq i,\ k\leq n\};\\[1.5mm]
	S_{3}:=&\{(i,k):\ i\in I,\  k\not\in I,\  1\leq i,\ k\leq n\}.
	\end{split}
	\end{equation*}
	Note $S_{j}\subset \{(i,k):\ i\neq k,\  1\leq i,\ k\leq n\}$ and $S_{j}\cap S_{l}=\emptyset$, for $1\leq j,\  l\leq 3$ and $j\neq l$.
	By the symmetry of $i, k$, we have
	\begin{equation*}\label{}
		\begin{split}
	\sum_{i\neq k}\frac{\lambda_{i}+\lambda_{k}}{(1+\lambda_{i}^{2})(1+\lambda_{k}^{2})}|V_{1}(\tilde{g}_{i\bar{k}})|^{2}
	=&2\sum_{i\neq k}\frac{\lambda_{k}}{(1+\lambda_{i}^{2})(1+\lambda_{k}^{2})}|V_{1}(\tilde{g}_{i\bar{k}})|^{2}\\
	\geq&2\sum_{l=1}^{3}\sum_{S_{l}}\frac{\lambda_{k}}{(1+\lambda_{i}^{2})(1+\lambda_{k}^{2})}|V_{1}(\tilde{g}_{i\bar{k}})|^{2}\\
	=&2\sum_{S_{1}}\frac{\lambda_{i}+\lambda_{k}}{(1+\lambda_{i}^{2})(1+\lambda_{k}^{2})}|V_{1}(\tilde{g}_{i\bar{k}})|^{2}~(\textrm{reverse}~ i,k~\textrm{in}~ S_{3})\\
	&+2\sum_{S_{2}}\frac{\lambda_{k}}{(1+\lambda_{i}^{2})(1+\lambda_{k}^{2})}|V_{1}(\tilde{g}_{i\bar{k}})|^{2}.
	\end{split}
	\end{equation*}
Hence, 	to prove \eqref{3..27}, we only need to prove
	\begin{equation}\label{3.26}
	\frac{\lambda_{i}+\lambda_{k}}{(1+\lambda_{i}^{2})(1+\lambda_{k}^{2})}\geq \frac{1}{(1+\lambda_{i}^{2})\lambda_{k}},\  (i,k)\in S_{1}
	\end{equation}
	and
	\begin{equation}\label{3.27}
	\frac{\lambda_{k}}{(1+\lambda_{i}^{2})(1+\lambda_{k}^{2})}\geq \frac{C_{B}}{\ve\mu_{1}^{3}}\frac{1}{1+\lambda_{i}^{2}},~\  (i,k)\in S_{2}.
	\end{equation}

Now we give the proof of the inequality (\ref{3.26}).	Note
	\begin{equation}\label{}
	\begin{split}
	\arctan \lambda_{j}+\arctan \lambda_{n}=&h-\sum_{i\neq j, n}\arctan\lambda_{i}\\
	\geq &(n-1)\frac{\pi}{2} -(n-2)\frac{\pi}{2} =\frac{\pi}{2}.
	\end{split}
	\end{equation}
	Hence
	\[\arctan \lambda_{j}\geq\frac{\pi}{2}-\arctan \lambda_{n}=\arctan \frac{1}{\lambda_{n}}.\]
	By the monotonicity of $\arctan x $, we have
		\begin{equation}\label{3.29}
	\lambda_{j}\lambda_{n}\geq 1,\ \ j=1,2, \cdots, n-1,
	\end{equation}
which implies (\ref{3.26}).
	
	Next, 	for (\ref{3.27}), it suffices to prove
	\begin{equation}\label{3.31}
	\frac{\lambda_{k}}{1+\lambda_{k}^{2}}\geq \frac{C_{B}}{\ve\mu_{1}^{3}}, \ \ k=1,2, \cdots, n.
	\end{equation}
Since $h>(n-1)\frac{\pi}{2}$, we have
	\[\arctan \lambda_{n}\geq h-\sum_{1\leq i\leq n-1}\arctan\lambda_{i}\geq h-(n-1)\frac{\pi}{2}\geq C_{1}^{-1}>0\]
	for some uniform constant $C_{1}$ depending on $\inf_{M}h$. This implies
	\begin{equation}\label{lower bound of eigenvalues}
	\lambda_{k}\geq \lambda_{n}\geq C_{1}^{-1}.
	\end{equation}
	Therefore, when $\lambda_{1}\gg 1$, we have \[\frac{\lambda_{k}}{1+\lambda_{k}^{2}}\geq \frac{1}{C_{1}(1+\lambda_{k}^{2})} \geq \frac{1}{C_{1}\mu_{1}^{2}}.\]
	Now if $\mu_{1}\geq C_{B}/\ve$, then we obtain \eqref{3.31}.
	It follows \eqref{3.27}. We complete the proof of \eqref{3..27}
	
Now we deal with the first term on the right hand side in  \eqref{third order term from concacity}.	By \eqref{3.31}, we have
	\begin{equation}\label{}
	\sum_{i}\frac{2\lambda_{i}}{(1+\lambda_{i}^{2})^{2}}|V_{1}(\tilde{g}_{i\bar{i}})|^{2}\geq \frac{C_{B}}{\ve\mu_{1}^{3}}\sum_{i\not\in I}\frac{1}{1+\lambda_{i}^{2}}{|V_{1}(\tilde{g}_{i\bar{i}})|^{2}}.
	\end{equation}
	Therefore, by \eqref{third order term from concacity} and \eqref{3..27}, we have
	\begin{equation}
	\begin{split}
	&-\frac{1}{\mu_{1}} F^{i\bar{k},j\bar{l}}V_{1}(\tilde{g}_{i\bar{k}})V_{1}(\tilde{g}_{j\bar{l}})\\
	\geq &
	\frac{2}{\mu_{1}}\sum_{i\not\in I, k\in I}F^{i\bar{i}}\tilde{g}^{k\bar{k}}|V_{1}(\tilde{g}_{i\bar{k}})|^{2}
	+\frac{C_{B}}{\ve\mu_{1}^{4}}\sum_{i\not\in I,k\not\in I}F^{i\bar{i}}{|V_{1}(\tilde{g}_{i\bar{k}})|^{2}}.\\
	\end{split}
	\end{equation}
\end{proof}

Define a new (1,0) vector field by
\[
\widetilde{e}_{1}=\frac{1}{\sqrt{2}}(V_{1}-\sqrt{-1}JV_{1}).
\]
At $x_{0}$, we can find a sequence of complex numbers $\nu_{1},\cdots, \nu_{n}$ such that
\begin{equation}\label{tilde e}
\widetilde{e}_{1}:=\sum_{1}^{n}\nu_{k}e_{k},~\sum_{1}^{n}|\nu_{k}|^{2}=1.
\end{equation}

\begin{lemma}\label{nu}
	 We have
	\[
	|\nu_{k}|\leq \frac{C_{B}}{\mu_{1}} ~\textrm{for all} ~k\not\in I.
	\]\
\end{lemma}
\begin{proof}
	The idea of the proof is similar to the argument of Lemma 5.6 in \cite{CTW16}. Since Case 2 does not hold, then we obtain
	\[
	\frac{\phi'}{4}\sum_{i\not\in I}\sum_{j} F^{i\bar{i}}(|e_{i}e_{j}u|^{2}+|e_{i}\bar{e}_{j}u|^{2})\leq(6n^{2}\sup_{\om}|\nabla \widetilde{\eta}|^{2})B^{2}e^{2B\widetilde{\eta}}F^{n\bar{n}}.
	\]
	While $ F^{n\bar{n}}\leq B^{3}e^{2B\widetilde{\eta}}F^{i\bar{i}}$ for each $i\not\in I$,
	it follows that
	\[
	\sum_{\gamma=2p+1}^{2n}\sum_{\beta=1}^{2n}|\nabla^{2}_{\gamma\beta}u|\leq C_{B}.
	\]
	Therefore, $|\Phi_{\beta}^{\gamma}|\leq C_{B}$ for $2p+1\leq\gamma\leq 2n$, $1\leq \beta\leq 2n$. Since $\Phi(V_{1})=\mu_{1}V_{1}$, then
	\[
	|V_{1}^{\gamma}|=|\frac{1}{\mu_{1}}(\Phi(V_{1}))^{\gamma}|=\frac{1}{\mu_{1}}|
	\sum_{\beta}\Phi_{\beta}^{\gamma}V_{1}^{\beta}|\leq \frac{C_{B}}{\mu_{1}},~2p+1\leq\gamma\leq 2n.
	\]
	Then, by  \eqref{tilde e},
	$|\nu_{k}|\leq |V_{1}^{2k-1}|+|V^{2k}_{1}|\leq \frac{C_{B}}{\mu_{1}}, k\not\in I.$
\end{proof}

 Now we can estimate the first three terms in Lemma \ref{Lma4.2}.
 Since $JV_{1}$ is $\chi$-unit and $\chi$-orthogonal to $V_{1}$, then we can find real numbers $\xi_{2},\cdots,\xi_{2n}$ such that
\begin{equation}\label{JV1}
JV_{1}=\sum_{\beta>1}\xi_{\beta}V_{\beta}, ~\sum_{\beta>1}\xi_{\beta}^{2}=1~ \textrm{at $x_{0}$}.
\end{equation}
~\\

\begin{lemma}\label{Lma3.5}
For any constant $\tau>0$, we have
\[
\begin{split}
    & (2-\ve)\sum_{\beta>1}F^{i\bar{i}}\frac{|e_{i}(u_{V_{\beta}V_{1}})|^{2}}{\mu_{1}(\mu_{1}-\mu_{\beta})}
-\frac{1}{\mu_{1}} F^{i\bar{k},j\bar{l}}V_{1}(\tilde{g}_{i\bar{k}})V_{1}(\tilde{g}_{j\bar{l}})
     -(1+\ve)\sum_{i\not\in I}F^{i\bar{i}} \frac{|e_{i}(\mu_{1})|^{2}}{\mu_{1}^{2}}\\
   \geq &(2-\ve
    )\sum_{i\not\in I}\sum_{\beta>1}F^{i\bar{i}} \frac{|e_{i}(u_{V_{\beta}V_{1}})|^{2}}{\mu_{1}(\mu_{1}-\mu_{\beta})}
    +\sum_{k\in I}\sum_{i\not\in I}\frac{2}{\mu_{1}} F^{i\bar{i}}\tilde{g}^{k\bar{k}}|V_{1}(\tilde{g}_{i\bar{k}})|^{2}\\
    &-3\ve\sum_{i\not\in I}F^{i\bar{i}}\frac{|e_{i}(\mu_{1})|^{2}}{\mu_{1}^{2}}-
    2(1-\ve)(1+\tau)\tilde{g}_{\tilde{1}\bar{\tilde{1}}}\sum_{k\in I}\sum_{i\not\in I}F^{i\bar{i}}\tilde{g}^{k\bar{k}}\frac{|V_{1}(\tilde{g}_{i\bar{k}})|^{2}}{\mu_{1}^{2}}\\
    &-\frac{C}{\ve}\mathcal{F}-(1-\ve)(1+\frac{1}{\tau})(\mu_{1}-\sum_{\beta>1}\mu_{\beta}
    \xi_{\beta}^{2})
    \sum_{i\not\in I}\sum_{\beta>1}\frac{F^{i\bar{i}}}{\mu_{1}^{2}}
    \frac{|e_{i}(u_{V_{\beta}V_{1}})|^{2}}{\mu_{1}-\mu_{\beta}}\\
\end{split}
\]
if we assume $\mu_{1}\geq \frac{C_{B}}{\ve}$, where $\tilde{g}_{\tilde{1}\bar{\tilde{1}}}=\sum \tilde{g}_{i\bar{i}}|\nu_{i}|^{2}$.
\end{lemma}
\textit{Proof.} First,  we can prove
\begin{equation}\label{4.26}
e_{i}(u_{V_{1}V_{1}})=\sqrt{2}\sum_{k} \bar{\nu}_{k}V_{1}(\tilde{g}_{i\bar{k}})-\sqrt{-1}\sum_{\beta>1}\xi_{\beta}e_{i}(u_{V_{1}
V_{\beta}})+O(\mu_{1}),\end{equation}
where $O(\mu_{1})$ denotes the terms which can be controlled by $\mu_{1}$. Indeed, since $\overline{\widetilde{e}}_{1}=\frac{1}{\sqrt{2}}(V_{1}+\sqrt{-1}JV_{1})$,
\[
e_{i}(u_{V_{1}V_{1}})=\sqrt{2} e_{i}(u_{V_{1}\overline{\widetilde{e}}_{1}})-\sqrt{-1} e_{i}(u_{V_{1}JV_{1}}).
\]
 For the first term, using $\tilde{g}_{i\bar{k}}=g_{i\bar{k}}+u_{i\bar{k}}$,
\begin{equation}\label{4.27}
\begin{split}
e_{i}(u_{V_{1}\overline{\widetilde{e}}_{1}})= & e_{i}(V_{1}\overline{\widetilde{e}}_{1}u-(\nabla_{V_{1}}\overline{\widetilde{e}}_{1})u)
=\overline{\widetilde{e}}_{1}e_{i}V_{1}u+O(\mu_{1})   \\
= & \sum_{k}\overline{\nu}_{k}V_{1}(\tilde{g}_{i\bar{k}})+O(\mu_{1}).
\end{split}
\end{equation}
For the second term, by \eqref{JV1},
\begin{equation}\label{4.28}
\begin{split}
e_{i}(u_{V_{1}JV_{1}})=&e_{i}{V_{1}JV_{1}}(u)+O(\mu_{1})=JV_{1}e_{i}{V_{1}}(u)+O(\mu_{1})  \\
=&\sum_{\beta>1}\xi_{\beta}V_{\beta}e_{i}{V_{1}}(u)+O(\mu_{1})=\sum_{\beta>1}\xi_{\beta}e_{i}(u_{V_{\beta}{V_{1}}})+O(\mu_{1}).\\
\end{split}
\end{equation}
Thus, (\ref{4.26}) follows from (\ref{4.27}) and (\ref{4.28}).
 Hence, by \eqref{4.26}, Lemma \ref{nu} and Cauchy-Schwarz inequality, we have
\begin{equation}\label{bad term 1 of third order derivative}
\begin{split}
&-(1+\ve)\sum_{i\not\in I} F^{i\bar{i}}\frac{|e_{i}(\mu_{1})|^{2}}{\mu_{1}^{2}}\\
=&-(1-2\ve)\sum_{i\not\in I} F^{i\bar{i}}\frac{|e_{i}(\sqrt{2}\sum_{k} \bar{\nu}_{k}V_{1}(\tilde{g}_{i\bar{k}})-\sqrt{-1}\sum_{\beta>1}\xi_{\beta}e_{i}(u_{V_{1}
V_{\beta}})+O(\mu_{1}))|^{2}}{\mu_{1}^{2}}\\
&-3\ve\sum_{i\not\in I} F^{i\bar{i}}\frac{|e_{i}(\mu_{1})|^{2}}{\mu_{1}^{2}}\\
\geq &
    -(1-\ve)\sum_{i\not\in I} F^{i\bar{i}}\frac{|\sqrt{2}\sum_{k\in I} \overline{\nu}_{k}V_{1}(\tilde{g}_{i\bar{k}})-\sqrt{-1}\sum_{\beta>1}\xi_{\beta}e_{i}(u_{V_{1}
V_{\beta}})|^{2}}{\mu_{1}^{2}}\\
&-3\ve \sum_{i\not\in I} F^{i\bar{i}}\frac{|e_{i}(\mu_{1})|^{2}}{\mu_{1}^{2}}-\frac{C_{B}}{\ve}\sum_{i\not\in I,k\not\in I}F^{i\bar{i}}\frac{|V_{1}(\tilde{g}_{i\bar{k}})|^{2}}{\mu_{1}^{4}}-\frac{C}{\ve}\mathcal{F}.
\end{split}
\end{equation}
~\\
In addition, using the Cauchy-Schwarz inequality, we have
\[
\Big|\sum_{\beta>1}\xi_{\beta}e_{i}(u_{V_{1}V_{\beta}})\Big|^{2}\leq \sum_{\beta>1}(\mu_{1}-\mu_{\beta}\xi_{\beta}^{2})
\sum_{\beta>1}\frac{|e_{i}(u_{V_{1}V_{\beta}})|^{2}}{\mu_{1}-\mu_{\beta}},
\]
and
\begin{equation}
\begin{split}
\Big|\sum_{k\in I}\overline{\nu}_{k}V_{1}(\tilde{g}_{i\bar{k}})\Big|^{2}&\leq \Big(\sum_{i}\tilde{g}_{i\bar{i}}|\nu_{i}|^{2}\Big)\sum_{k\in I}\tilde{g}^{k\bar{k}}|V_{1}(\tilde{g}_{i\bar{k}})|^{2}\\
&= \tilde{g}_{\tilde{1}\bar{\tilde{1}}}\sum_{k\in I}\tilde{g}^{k\bar{k}}|V_{1}(\tilde{g}_{i\bar{k}})|^{2}.\\
\end{split}
\end{equation}
Then for each $\gamma>0$, using the Cauchy-Schwarz inequality again, we get
\[
\begin{split}
   & (1-\ve)\sum_{i\not\in I} F^{i\bar{i}}\frac{|\sqrt{2}\sum_{k\in I} \overline{\nu_{k}}V_{1}(\tilde{g}_{i\bar{k}})-\sqrt{-1}\sum_{\beta>1}\xi_{\beta}e_{i}(u_{V_{1}
V_{\beta}})|^{2}}{\mu_{1}^{2}}\\
\leq & 2(1-\ve)(1+\tau)\sum_{i\not\in I} F^{i\bar{i}}\frac{|\sum_{k\in I} \overline{\nu_{k}}V_{1}(\tilde{g}_{i\bar{k}})|^{2}}{\mu_{1}^{2}}\\
&+(1-\ve)(1+\frac{1}{\tau})\sum_{i\not\in I} F^{i\bar{i}}\frac{|\sum_{\beta>1}\xi_{\beta}e_{i}(u_{V_{1}
V_{\beta}})|^{2}}{\mu_{1}^{2}}\\
\leq & 2(1-\ve)(1+\tau)\tilde{g}_{\tilde{1}\bar{\tilde{1}}}\sum_{i\not\in I}\sum_{k\in I} \frac{F^{i\bar{i}}}{\mu_{1}^{2}} \tilde{g}^{k\bar{k}}|V_{1}(\tilde{g}_{i\bar{k}})|^{2}\\
&+(1-\ve)(1+\frac{1}{\tau})(\mu_{1}-\sum_{\beta>1} \mu_{\beta}\xi_{\beta}^{2})\sum_{i\not\in I}\sum_{\beta>1}\frac{F^{i\bar{i}}}{\mu_{1}^{2}} \frac{|e_{i}(u_{V_{\beta}V_{1}})|^{2}}{\mu_{1}-\mu_{\beta}}.
\end{split}
\]
Combining with \eqref{bad term 1 of third order derivative}, we have
\begin{equation*}
\begin{split}
&-(1+\ve)\sum_{i\not\in I} F^{i\bar{i}}\frac{|e_{i}(\mu_{1})|^{2}}{\mu_{1}^{2}}\\
\geq &-2(1-\ve)(1+\tau)\tilde{g}_{\tilde{1}\bar{\tilde{1}}}\sum_{i\not\in I}\sum_{k\in I} \frac{F^{i\bar{i}}}{\mu_{1}^{2}} \tilde{g}^{k\bar{k}}|V_{1}(\tilde{g}_{i\bar{k}})|^{2}\\
&-(1-\ve)(1+\frac{1}{\tau})(\mu_{1}-\sum_{\beta>1} \mu_{\beta}\xi_{\beta}^{2})\sum_{i\not\in I}\sum_{\beta>1}\frac{F^{i\bar{i}}}{\mu_{1}^{2}} \frac{|e_{i}(u_{V_{\beta}V_{1}})|^{2}}{\mu_{1}-\mu_{\beta}}\\
&-3\ve \sum_{i\not\in I} F^{i\bar{i}}\frac{|e_{i}(\mu_{1})|^{2}}{\mu_{1}^{2}}-\frac{C_{B}}{\ve}\sum_{i\not\in I,k\not\in I}F^{i\bar{i}}\frac{|V_{1}(\tilde{g}_{i\bar{k}})|^{2}}{\mu_{1}^{4}}-\frac{C}{\ve}\mathcal{F}.
\end{split}
\end{equation*}


Then the lemma follows from it and Lemma \ref{lower bound of concave}.
\qed

\begin{lemma}\label{Lma3.6}
If we assume $\mu_{1}\geq C/{\ve^{3}}$, then
\[
\begin{split}
&(2-\ve)\sum_{\beta>1}F^{i\bar{i}}\frac{|e_{i}(u_{V_{\beta}V_{1}})|^{2}}
{\mu_{1}(\mu_{1}-\mu_{\beta})}-\frac{1}{\mu_{1}} F^{i\bar{k},j\bar{l}}V_{1}(\tilde{g}_{i\bar{k}})V_{1}(\tilde{g}_{j\bar{l}})
-(1+\ve)\sum_{i\not\in I} F^{i\bar{i}}\frac{|e_{i}(\mu_{1})|^{2}}{\mu_{1}^{2}}\\
\geq& -6\ve B^{2}e^{2B\widetilde{\eta}}\sum_{i} F^{i\bar{i}}|e_{i}(\widetilde{\eta})|^{2}-6\ve (\phi')^{2}\sum_{i\not\in I}F^{i\bar{i}}|e_{i}(|\nabla u|^{2})|^{2}-\frac{C}{\ve}\mathcal{F}.
\end{split}
\]
\end{lemma}
\begin{proof} By \eqref{5..10}, it suffices to prove
\begin{equation}\label{4.30}
\begin{split}
(2-\ve)&\sum_{\beta>1}F^{i\bar{i}}\frac{|e_{i}(u_{V_{\beta}V_{1}})|^{2}}
{\mu_{1}(\mu_{1}-\mu_{\beta})}-\frac{1}{\mu_{1}} F^{i\bar{k},j\bar{l}}V_{1}(\tilde{g}_{i\bar{k}})V_{1}(\tilde{g}_{j\bar{l}})\\&
-(1+\ve)\sum_{i\not\in I} F^{i\bar{i}}\frac{|e_{i}(\mu_{1})|^{2}}{\mu_{1}^{2}}
\geq -3\ve \sum_{i\not\in I} F^{i\bar{i}}\frac{|e_{i}(\mu_{1})|^{2}}{\mu_{1}^{2}}-\frac{C}{\ve} \mathcal{F}.
\end{split}
\end{equation}
We divide the proof into two cases.
~\\
\textbf{Case I:} Assume that
\begin{equation}\label{5.32}
\mu_{1}+\sum_{\beta>1}\mu_{\beta}\xi_{\beta}^{2}\geq 2(1-\ve)\tilde{g}_{\tilde{1}\bar{\tilde{1}}}>0.
\end{equation}
 It follows from Lemma \ref{Lma3.5} and (\ref{5.32}) that
\begin{equation}\label{5.76}
\begin{split}
    & (2-\ve)\sum_{\beta>1}F^{i\bar{i}}\frac{|e_{i}(u_{V_{\beta}V_{1}})|^{2}}{\mu_{1}(\mu_{1}-\mu_{\beta})}
-\frac{1}{\mu_{1}} F^{i\bar{k},j\bar{l}}V_{1}(\tilde{g}_{i\bar{k}})V_{1}(\tilde{g}_{j\bar{l}})\\
    & -(1+\ve)\sum_{i\not\in I} F^{i\bar{i}}\frac{|e_{i}(\mu_{1})|^{2}}{\mu_{1}^{2}}\\
    \geq &\sum_{i\not\in I}\sum_{\beta>1} \frac{F^{i\bar{i}}}{\mu_{1}}\left(\frac{(2-\ve
    )\mu_{1}}{\mu_{1}-\mu_{\beta}}|e_{i}(u_{V_{\beta}V_{1}})|^{2}\right)
    +\sum_{k\in I}\sum_{i\not\in I}\frac{2}{\mu_{1}} F^{i\bar{i}}\tilde{g}^{k\bar{k}}|V_{1}(\tilde{g}_{i\bar{k}})|^{2}\\
    &-3\ve\sum_{i\not\in I}F^{i\bar{i}}\frac{|e_{i}(\mu_{1})|^{2}}{\mu_{1}^{2}}-
    (1+{\tau})(\mu_{1}+\sum_{\beta>1}\mu_{\beta}\xi_{\beta}^{2})\sum_{k\in I}\sum_{i\not\in I}F^{i\bar{i}}\tilde{g}^{k\bar{k}}\frac{|V_{1}(\tilde{g}_{i\bar{k}})|^{2}}{\mu_{1}^{2}}\\
    &-\frac{C}{\ve}\mathcal{F}-
    (1-\ve)(1+\frac{1}{\tau})(\mu_{1}-\sum_{\beta>1}\mu_{\beta}\xi_{\beta}^{2})
    \sum_{i\not\in I}\sum_{\beta>1}\frac{F^{i\bar{i}}}{\mu_{1}^{2}}\frac{|e_{i}(u_{V_{\beta}V_{1}})|^{2}}{\mu_{1}-\mu_{\beta}}.
\end{split}
\end{equation}
Choose
\[
\tau=\frac{\mu_{1}-\underset{\beta>1}{\sum}\mu_{\beta}\xi_{\beta}^{2}}
{\mu_{1}+\underset{\beta>1}{\sum}\mu_{\beta}\xi_{\beta}^{2}}.
\]
Therefore,
\[
\begin{split}
    & (1+{\tau})(\mu_{1}+\sum_{\beta>1}\mu_{\beta}\xi_{\beta}^{2})\sum_{k\in I}\sum_{i\not\in I}F^{i\bar{i}}\tilde{g}^{k\bar{k}}\frac{|V_{1}(\tilde{g}_{i\bar{k}})|^{2}}{\mu_{1}^{2}} \\
    & +(1-\ve)(1+\frac{1}{\tau})(\mu_{1}-\sum_{\beta>1}\mu_{\beta}\xi_{\beta}^{2})
    \sum_{i\not\in I}\sum_{\beta>1}\frac{F^{i\bar{i}}}{\mu_{1}^{2}}
    \frac{|e_{i}(u_{V_{\beta}V_{1}})|^{2}}{\mu_{1}-\mu_{\beta}}\\
    =& 2\sum_{k\in I}\sum_{i\not\in I}F^{i\bar{i}}\tilde{g}^{k\bar{k}}\frac{|V_{1}(\tilde{g}_{i\bar{k}})|^{2}}{\mu_{1}}
    +2(1-\ve)\sum_{i\not\in I}\sum_{\beta>1}\frac{F^{i\bar{i}}}{\mu_{1}}
    \frac{|e_{i}(u_{V_{\beta}V_{1}})|^{2}}{\mu_{1}-\mu_{\beta}}.
\end{split}
\]
Then (\ref{4.30}) follows from (\ref{5.76}).\qed
~\\
\textbf{Case II:}  Assume that
\begin{equation}\label{4.31-1}
{\mu_{1}+\sum_{\beta>1}\mu_{\beta}\xi_{\beta}^{2}}< 2(1-\ve)\tilde{g}_{\tilde{1}\bar{\tilde{1}}}.
\end{equation}
By a directly calculation,
\begin{equation}\label{4.32}
\begin{split}
&\tilde{g}(\tilde{e},\overline{\tilde{e}})\\
=&    g(\tilde{e},\overline{\tilde{e}})+\tilde{e}\overline{\tilde{e}}(u)-[\tilde{e},\overline{\tilde{e}}]^{(0,1)}(u)\\
=&g(\tilde{e},\overline{\tilde{e}})+\frac{1}{2}(V_{1}V_{1}(u)+(JV_{1})(JV_{1})(u)+\sqrt{-1}[V_{1},JV_{1}](u))\\
&\quad\, -[\tilde{e},\overline{\tilde{e}}]^{(0,1)}(u)\\
=&   \frac{1}{2}\big(\mu_{1}+\sum_{\alpha>1}\mu_{\alpha}\xi_{\alpha}^{2}\big)+g(\tilde{e},\overline{\tilde{e}})+(\nabla_{V_{1}}V_{1})(u)+(\nabla_{JV_{1}}JV_{1})(u)\\
&  \quad\,   +\sqrt{-1}[V_{1},JV_{1}](u) -[\tilde{e},\overline{\tilde{e}}]^{(0,1)}(u)\\
\leq& \frac{1}{2}\big(\mu_{1}+\sum_{\alpha>1}\mu_{\alpha}\xi_{\alpha}^{2}\big)+C.
\end{split}
\end{equation}
Plugging (\ref{4.31-1}) into (\ref{4.32}), then
\begin{equation}\label{3..42}
\tilde{g}_{\tilde{1}\bar{\tilde{1}}}\leq C/{\ve}.
\end{equation}
By \eqref{4.32} we have $\mu_{1}+\sum_{\beta>1} \mu_{\beta}\xi_{\beta}^{2}\geq -C$.
Hence,
\[
0<\mu_{1}-\sum_{\beta>1} \mu_{\beta}\xi_{\beta}^{2}\leq 2\mu_{1}+C\leq (2+2\ve^{2})\mu_{1}
\]
provided by $\mu_{1}\geq C/{\ve^{2}}$.  Choose $\tau=1/{\ve^{2}}$. It follows that
\[
\begin{split}
(1-\ve)(1+\frac{1}{\tau})(\mu_{1}-\sum_{\beta>1}\mu_{\beta}\xi_{\beta}^{2})
 \leq & 2(1-\ve)(1+\ve^{2})^{2}\mu_{1}\\
 \leq &(2-\ve)\mu_{1},
 \end{split}
\]
when $\ve$ small enough.
Then, by Lemma \ref{Lma3.5}, we have
\[
\begin{split}
    & (2-\ve)\sum_{\beta>1}F^{i\bar{i}}\frac{|e_{i}(u_{V_{\beta}V_{1}})|^{2}}{\mu_{1}(\mu_{1}-\mu_{\beta})}
-\frac{1}{\mu_{1}} F^{i\bar{k},j\bar{l}}V_{1}(\tilde{g}_{i\bar{k}})V_{1}(\tilde{g}_{j\bar{l}}) -(1+\ve)\sum_{i\not\in I} F^{i\bar{i}}\frac{|e_{i}(\mu_{1})|^{2}}{\mu_{1}^{2}}\\
    \geq &2\sum_{k\in I}\sum_{i\not\in I} F^{i\bar{i}}\tilde{g}^{k\bar{k}}\frac{|V_{1}(\tilde{g}_{i\bar{k}})|^{2}}{\mu_{1}}-3\ve\sum_{i\not\in I}F^{i\bar{i}}\frac{|e_{i}(\mu_{1})|^{2}}{\mu_{1}^{2}}\\
    &-2(1-\ve)(1+\frac{1}{\ve^{2}})\tilde{g}_{\tilde{1}\bar{\tilde{1}}}\sum_{k\in I}\sum_{i\not\in I}F^{i\bar{i}}\tilde{g}^{k\bar{k}}\frac{|V_{1}(\tilde{g}_{i\bar{k}})|^{2}}{\mu_{1}^{2}}    -\frac{C}{\ve}\mathcal{F}\\
    \stackrel{\eqref{3..42}}{\geq} &2\sum_{k\in I}\sum_{i\not\in I} F^{i\bar{i}}\tilde{g}^{k\bar{k}}\frac{|V_{1}(\tilde{g}_{i\bar{k}})|^{2}}{\mu_{1}} -3\ve\sum_{i\not\in I}F^{i\bar{i}}\frac{|e_{i}(\mu_{1})|^{2}}{\mu_{1}^{2}}\\
    &-(1-\ve)(1+\frac{1}{\ve^{2}})\frac{C}{\ve}\sum_{k\in I}\sum_{i\not\in I}F^{i\bar{i}}\tilde{g}^{k\bar{k}}\frac{|V_{1}(\tilde{g}_{i\bar{k}})|^{2}}{\mu_{1}^{2}}    -\frac{C}{\ve}\mathcal{F}\\
    \geq &-3\ve\sum_{i\not\in I}F^{i\bar{i}}\frac{|e_{i}(\mu_{1})|^{2}}{\mu_{1}^{2}}-\frac{C}{\ve}\mathcal{F},\\
\end{split}
\]
~\\
if we assume $\mu_{1}\geq C/{\ve^{3}}$ in the last inequality. This proves  (\ref{4.30}).
\end{proof}

We now complete the proof of second order estimates. By Lemma \ref{Lma3.6}, Lemma \ref{partial third order} and (\ref{4.7"}), we have
\begin{equation}\label{}
\begin{split}
L(Q)  \geq & -6\ve B^{2}e^{2B\widetilde{\eta}} F^{i\bar{i}}|e_{i}(\widetilde{\eta})|^{2}-6\ve (\phi')^{2}\sum_{i\not\in I}F^{i\bar{i}}|e_{i}(|\nabla u|^{2})|^{2}-\frac{C}{\ve}\mathcal{F}\\
&+ \frac{\phi'}{2}\sum_{j} F^{i\bar{i}}(|e_{i}e_{j}u|^{2}+|e_{i}\bar{e}_{j}u|^{2})
+B^{2}e^{B\widetilde{\eta}} F^{i\bar{i}}|e_{i}(\widetilde{\eta})|^{2}+Be^{B\widetilde{\eta}}L(\widetilde{\eta})\\
& +\phi''F^{i\bar{i}}|e_{i}(|\nabla u|^{2})|^{2}-{2(\phi')^{2}\sum_{i\in I} F^{i\bar{i}}|e_{i}(|\nabla u|^{2})|^{2}}.
\end{split}
\end{equation}
Choose $\ve< \min\{\frac{1}{6n},\theta/6\}$ such that $e^{B\widetilde{\eta}(0)}=\frac{1}{6\ve}$. By $\phi''=2(\phi')^{2}$, then
\[
\begin{split}
0 \geq &-\frac{C}{\ve}\mathcal{F}+ \frac{\phi'}{2}\sum_{j} F^{i\bar{i}}(|e_{i}e_{j}u|^{2}+|e_{i}\bar{e}_{j}u|^{2})\\
&+(B^{2}e^{B\widetilde{\eta}}-6\ve B^{2}e^{2B\widetilde{\eta}}) \sum_{i\not\in I}F^{i\bar{i}}|e_{i}(\widetilde{\eta})|^{2}+Be^{B\widetilde{\eta}}L(\widetilde{\eta})\\
=&-\frac{C}{\ve}\mathcal{F}+ \frac{\phi'}{2}\sum_{j} F^{i\bar{i}}(|e_{i}e_{j}u|^{2}+|e_{i}\bar{e}_{j}u|^{2})+Be^{B\widetilde{\eta}}L(\widetilde{\eta}).
\end{split}
\]
In other words,
\begin{equation}\label{3.44}
\frac{B}{6\ve}L(\widetilde{\eta})-\frac{C}{\ve}\mathcal{F}+ \frac{\phi'}{2}\sum_{j} F^{i\bar{i}}(|e_{i}e_{j}u|^{2}+|e_{i}\bar{e}_{j}u|^{2})\leq 0.
\end{equation}
(a). Suppose (\ref{2..25}) holds. Then we have
\[
(\frac{B\theta}{6\ve}-\frac{C}{\ve})\mathcal{F}+ \frac{\phi'}{2}\sum_{j} F^{i\bar{i}}(|e_{i}e_{j}u|^{2}+|e_{i}\bar{e}_{j}u|^{2})\leq 0.
\]
~\\
Choose $B$ sufficiently large and $\ve<\theta/6$ small enough  such that $B\theta/6-C\geq B\ve$. Then at $x_{0}$ we have
\[
0\geq B\mathcal{F}+ \frac{\phi'}{2}\sum_{j} F^{i\bar{i}}(|e_{i}e_{j}u|^{2}+|e_{i}\bar{e}_{j}u|^{2}).
\]
This yields a contradiction.
~\\
(b). Suppose  (\ref{2..26}) holds. That is,
\[1\geq F^{i\bar{i}}\geq \theta\mathcal{F}\geq \theta\mathcal{K},~ \textrm{for}~ i=1,2,\cdots,n.\]
Combining  \eqref{lower bound of Leta} with \eqref{3.44},  we have
\[
 \frac{\theta\phi'}{2}\mathcal{K}\sum_{i,j} (|e_{i}e_{j}u|^{2}+|e_{i}\bar{e}_{j}u|^{2})\leq \frac{nB}{12\ve}+\frac{C}{\ve}.
\]
Therefore,
\begin{equation}\label{}
\sum_{i,j} (|e_{i}e_{j}u|^{2}+|e_{i}\bar{e}_{j}u|^{2})\leq C_{B}.
\end{equation}
Then this proves the Case 3. In conclusion, we obtain the second order estimates.
\qed
\bigskip

By \eqref{the first order derivative of F} and Theorem \ref{Thm4.1}, the equation \eqref{DHYM} is uniformly elliptic.  By the $C^{2,\alpha}$ estimates (e.g. \cite[Theorem 1.1]{TWWY15}), it follows
 $\|u\|_{C^{2,\alpha}}\leq C.$
Then, by a standard bootstrapping argument, we complete the proof of Theorem \ref{main theorem}.

\section{Proof of Theorem \ref{constant angle theorem}}
We use the arguments of \cite{CJY} and \cite{SW} to prove the Theorem \ref{constant angle theorem}. We will use the continuity method. First we give the openness.
\subsection{Openness}
Consider the family of  equations
\begin{equation}\label{cm1}
\sum_{i}\arctan\lambda_{i}(u_{t})=(1-t)h_{0}+th_{1}+c_{t},
\end{equation}
where $h_{0}, h_{1}\in ((n-1)\frac{\pi}{2}, n\frac{\pi}{2} )$ are smooth functions on $M$.
\begin{prop}\label{open}
	Suppose $u_{t_{0}}$ satisfies
	\[
	\sum_{i}\arctan\lambda_{i}(u_{t_{0}})=(1-t_{0})h_{0}+t_{0}h_{1}+c_{t_{0}}.
	\]
Then there exists $\epsilon>0$ such that when $|t-t_{0}|\leq \epsilon$, we can find $(u_{t}, c_{t})$ solving \eqref{cm1}.  	
\end{prop}
\begin{proof}
	Set	$L=F^{i\bar{j}}(e_{i}\bar{e}_{j}-[e_{i},\bar{e}_{j}]^{0,1}).$
		Since the operator is homotopic to the  canonical Laplacian operator $\Delta^{C}$, the index of $L$ is zero, where $\Delta^{C} \psi=\frac{n\chi^{n-1}\wedge \ddbar \psi}{\chi^{n}}$. By the  maximum principle,
	\begin{equation}\label{ker L}
	\text{Ker}(L)=\{\text{Constants} \}.
	\end{equation}
	Denote by $L^{*}$  the $L^{2}$-adjoint operator of $L$. By the Fredholm theorem, there is a smooth function $\vp_{0}$ such that
	\begin{equation}\label{ker L*}
	\text{Ker}(L^{*})=\text{Span}\{\vp_{0}\}.
	\end{equation}
	Now we prove that $\vp_{0}$ does not change the sign. If $\vp_{0}$ changes sign, there is a positive function $\tilde{f}$ perpendicular to it, but
	cannot be the image of $L$ by the maximum principle. It is a contradiction. Without
	loss of generality, assume that  $\vp_{0}$ is non-negative.
	By the strong maximum principle for the elliptic operator $L^{*}$, the function $\vp_{0}>0$.	Denote $$\tilde{B}=\Big\{\psi\in C^{2,\alpha}(M): \int_{M}\psi \vp_{0}\chi^{n}=0 \text{\ and\ } \omega+\ddbar\psi>0 \Big\}\times \mathbb{R}.$$
	Consider the map $G$, from $\tilde{B}$ to $C^{\alpha}(M)$, where
		\begin{equation*}
		G(\psi,c):=\sum_{i}\arctan\lambda_{i}(\psi)-c.
	\end{equation*}
Then	the linear operator of $G$ at $(u_{0}, 0)$ is
\begin{equation}\label{linear operator}
L-c: \left\{\xi \in C^{2,\alpha}(M): \int_{M}\xi \vp_{0}\chi^{n}=0\right\}\times \mathbb{R}\rightarrow  C^{\alpha}(M).
\end{equation}
For any $\hat{h}\in  C^{\alpha}(M)$, there exists a unique  constant $c$ such that $\int_{M}(\hat{h}+c)\vp_{0}\chi^{n}=0$. Then by \eqref{ker L*} and the Fredholm theorem, there exists $\xi$ such that
$L(\xi)-c=\hat{h}.$ Hence $L-c$ is surjective.
Let $(\xi_{0}, \hat{c})$ be the solution of $L-c=0$. By \eqref{ker L*} and the Fredholm theorem, $ \hat{c}=0$. Using $\eqref{ker L}$ and \eqref{linear operator}, we obtain $\xi_{0}=0.$
Hence, $ L-c$ is injective.
Then by the implicit function theory,  when $|t-t_{0}|$ small enough, there exist $u_{t}$ and a constant $c_{t}$ satisfying
 \begin{equation*}
 \sum_{i}\arctan\lambda_{i}(u_{t})=(1-t)h_{0}+th_{1}+c_{t}.
 \end{equation*}

%


%
%
%
%

\end{proof}

\subsection{Existence}
Suppose $\underline{u}$ is a $\mathcal{C}$-subsolution to the deformed Hermitian-Yang-Mills equation \eqref{DHYM}. Recall $\lambda_{i}(\hat{u})$ are the eigenvalues of $\omega_{\hat{u}}$. 
Denote $\theta_{0}=\sum_{i}\arctan\lambda_{i}(\hat{u}).$
Now we use the continuity method to prove that there exists a solution when the right hand side $h$ of \eqref{DHYM} is a constant.
\begin{prop}\label{existence c1}
Under the assumption of Theorem \ref{constant angle theorem}, there exists  a function $u$ on $M$ and a constant $c$ such that
	\begin{equation*}
\sum_{i}\arctan\lambda_{i}=h+c,
	\end{equation*}
where 	$h+c>\frac{(n-1)\pi}{2}$.
\end{prop}
\begin{proof}
	Consider the family of equations
\begin{equation}\label{c1}
\sum_{i}\arctan\lambda_{i}(u_{t})=(1-t)\theta_{0}+th+c_{t}.
\end{equation}
Define \[I=\{t\in[0,1]: \text{\ there exist\ } (u_{t}, c_{t})\in \tilde{B} \text{\ solving \ } \eqref{c1}\}.\]
Note $(\hat{u},0)$ is the solution of \eqref{c1} at $t=0$. Then $I$ is non-empty. By Proposition \ref{open}, $I$ is open.
To prove $I$ is closed, by \eqref{main theorem}, it suffices to prove $\underline{u}$ is still a $\mathcal{C}$-subsolution of \eqref{c1} for any $t\in[0,1]$ and
\begin{equation}\label{closed condition}
(1-t)\theta_{0}+th+c_{t}\geq\inf_{M} \theta_{0}>(n-1)\frac{\pi}{2}.
\end{equation}
First, assume $u_{t}-\hat{u}$ achieves its maximum at the point $q$. Then
\[\ddbar(u_{t}-\hat{u})(q)\leq0.\]
It follows that
\begin{equation*}
F(\omega_{u_{t}})(q)-F(\omega_{\hat{u}})(q)=\int_{0}^{1}F^{i\ol{j}}(\omega_{\hat{u}}+s\ddbar(u_{t}-\hat{u}))\ ds\, (u_{t}-\hat{u})_{i\ol{j}}(q)\leq0.
\end{equation*}
Then at $q$,
\begin{equation*}
\theta_{0}(q)\geq \sum_{i}\arctan\lambda_{i}(u_{t})(q)=(1-t)\theta_{0}(q)+th+c_{t},
\end{equation*}
which implies
\begin{equation}\label{upper bound constant}
c_{t}\leq t(\theta_{0}(q)-h)\leq 0.
\end{equation}
Here we used  $h\geq \theta_{0}.$
Then, we have
\begin{equation*}
\begin{split}
\sum_{i\neq j}\arctan(\lambda_{i}(\underline{u}))&\geq h-\frac{\pi}{2}\\
                                  &\geq(1-t)\theta_{0}+th+c_{t}-\frac{\pi}{2} .
\end{split}
\end{equation*}
By Lemma \ref{subsolution}, $\underline{u}$ is a $\mathcal{C}$-subsolution for $t\in[0,1]$.
Assume $u_{t}-\hat{u}$ achieves its minimum at the point $q'$.
Similarly with \eqref{upper bound constant}, we have
 \begin{equation}
c_{t}\geq- t(h-\theta_{0}(q'))\geq -\sup_{M} t(h-\theta_{0}).
\end{equation}
Assume $\theta_{0}$ achieves its minimum at the point $p$. Note that $h$ is a constant.
Then, we obtain
\begin{equation}\label{hyper c}
\begin{split}
\inf_{M}\left((1-t)\theta_{0}+th+c_{t}\right)&=(1-t)\theta_{0}(p)+th+c_{t}\\
&=\theta_{0}(p)+t(h-\theta_{0}(p))+c_{t}\\
&=\theta_{0}(p)+t\sup_{M}(h-\theta_{0})+c_{t}\\
&\geq\theta_{0}(p)= \inf_{M}\theta_{0}>(n-1)\frac{\pi}{2}.
\end{split}
\end{equation}
By Theorem \ref{main theorem}, we conclude $I$ is closed.
\end{proof}

\end{document}